\documentclass[12pt]{amsart}

\usepackage{amsmath}
\usepackage{amssymb}  % defines \nmid, for example
\usepackage{latexsym} % for \Box
\usepackage{comment}
\usepackage{url}
\usepackage{fullpage,url,amssymb,amsmath,amsthm,amsfonts,mathrsfs}
\usepackage[usenames,dvipsnames]{color}
\usepackage[pagebackref = true, colorlinks = true, linkcolor = blue, citecolor =
 Green]{hyperref}
\usepackage[alphabetic,lite,nobysame]{amsrefs}
\usepackage{enumitem}
\usepackage{amscd}   % for commutative diagrams
\usepackage[all, cmtip]{xy} % for complicated commutative diagrams
\usepackage{xfrac}
\usepackage[T1]{fontenc}
\usepackage[title]{appendix}

\DeclareFontEncoding{OT2}{}{} % to enable usage of cyrillic fonts
\newcommand{\textcyr}[1]{%
{\fontencoding{OT2}\fontfamily{cmr}\fontseries{m}\fontshape{n}\selectfont #1}}
\usepackage[all]{xy}
\usepackage{fullpage}
\newcommand{\Sha}{{\mbox{\textcyr{Sh}}}}

\usepackage{color} 
\newcommand{\defi}[1]{\textsf{#1}} % for defined terms

\def\act#1#2%
  {\mathop{}%
   \mathopen{\vphantom{#2}}^{#1}%
   \kern-3\scriptspace%
   #2}

\newcommand{\Z}{{\mathbb Z}}

\newcommand{\Q}{{\mathbb Q}}
\newcommand{\R}{{\mathbb R}}
\newcommand{\C}{{\mathbb C}}
\newcommand{\F}{{\mathbb F}}
\newcommand{\A}{{\mathbb A}}

\newcommand{\G}{{\mathbb G}}

\newcommand{\Kbar}{{\overline{K}}}

\newcommand{\kbar}{{\overline{k}}}
\newcommand{\Fbar}{{\overline{\F}}}

\newcommand{\calA}{{\mathcal A}}

\newcommand{\calO}{{\mathcal O}}

\newcommand{\calX}{{\mathcal X}}

\DeclareMathOperator{\Exc}{Exc}
\DeclareMathOperator{\Sel}{Sel}

\DeclareMathOperator{\End}{End}
\DeclareMathOperator{\Hom}{Hom}

\DeclareMathOperator{\Gal}{Gal}

\DeclareMathOperator{\Res}{Res}
\DeclareMathOperator{\Br}{Br}
\DeclareMathOperator{\Sym}{Sym}

\DeclareMathOperator{\Pic}{Pic}

\DeclareMathOperator{\HH}{H}

\DeclareMathOperator{\Spec}{Spec}
\DeclareMathOperator{\Aut}{Aut}

\DeclareMathOperator{\NS}{NS}

\newcommand{\res}{\operatorname{res}}
\def\sep{{\rm sep}}

\newtheorem{Theorem}{Theorem}[section]
\newtheorem{Lemma}[Theorem]{Lemma}
\newtheorem{Proposition}[Theorem]{Proposition}
\newtheorem{Corollary}[Theorem]{Corollary}
\newtheorem{Definition}[Theorem]{Definition}

\newtheorem{Remark}[Theorem]{Remark}
\newtheorem{Conjecture}[Theorem]{Conjecture}

\theoremstyle{definition}

\numberwithin{equation}{section}

\begin{document}
\title{Adelic Mordell-Lang and the Brauer-Manin Obstruction}

\author{Brendan Creutz}
\address{School of Mathematics and Statistics, University of Canterbury, Private
 Bag 4800, Christchurch 8140, New Zealand}
\email{brendan.creutz@canterbury.ac.nz}
\urladdr{http://www.math.canterbury.ac.nz/\~{}b.creutz}


\begin{abstract}
Let $X$ be a closed subvariety of an abelian variety $A$ over a global function field $k$ such that the base change of $A$ to an algebraic closure does not have any positive dimensional isotrivial quotient. We prove that every adelic point on $X$ which is the limit of a sequence of $k$-rational points on $A$ is a limit of $k$-rational points on $X$. Assuming finiteness of the Tate-Shafarevich group of $A$, this implies that the rational points on $X$ are dense in the Brauer set of $X$. Similar results are obtained over totally imaginary number fields, conditionally on an adelic Mordell-Lang conjecture.
\end{abstract}

\maketitle

\vspace{-5mm}
\section{Introduction}

Let $A$ be an abelian variety over a global field $k$. By the celebrated Mordell-Weil theorem the set $A(k)$ of rational points on $A$ is a finitely generated abelian group. Moreover, the standard arithmetic duality theorems describe, at least conjecturally, how $A(k)$ sits inside the set $A(\A_k)$ of adelic points. The topological closure $\overline{A(k)}$ of the image of $A(k)$ in the space $A(\A_k)_\bullet$ of connected components of $A(\A_k)$ is isomorphic to the profinite completion of $A(k)$ \cite{SerreII,MilneCongruence}. It is conjectured that the Tate-Shafarevich group $\Sha(A)$ of $A$ is finite and, in particular, that its maximal divisible subgroup is trivial. If $\Sha(A)_\textup{div}$ is trivial, then $\overline{A(k)}$ is equal to the set $A(\A_k)^{\Br}_\bullet$ of connected components of $A(\A_k)$ which are orthogonal to the Brauer group of $A$ \cite{Manin,Wang,PoonenVoloch}. We prove the converse of this last statement in Theorem~\ref{thm:notranscendental} below. 

For $X$ a closed subvariety of $A$, the Mordell-Lang conjecture describes $X(k)$ as the set of rational points on a finite union of special subvarieties. This was proved by Faltings \cite{Faltings} in the number field case and by Hrushovski \cite{Hrushovski} in the function field case (See also \cite{Voloch,AV,RosslerML}). However, it is less well understood how $X(k)$ sits inside $X(\A_k)$. The set $X(\A_k)^{\Br}$ of adelic points orthogonal to the Brauer group is closely related to the intersection of $X(\A_k)$ with the closure of the image of $A(k)$ in $A(\A_k)$. In the function field case it is conjectured that the sets $\overline{X(k)}$, $X(\A_k) \cap \overline{A(k)}$ and $X(\A_k)^{\Br}$ are all equal \cite[Conjecture C]{PoonenVoloch}. A similar statement with modifications at the archimedean primes is conjectured in the number field case. These conjectures originate from questions posed independently by Scharaschkin \cite{Scharaschkin} and Skorobogatov \cite[p. 133]{Skorobogatov} for curves over number fields. There is extensive numerical and theoretical evidence for these conjectures, though predominantly in cases where $X(k)$ is finite \cite{Flynn,Poonen,Stoll,BruinStoll,PoonenVoloch,BGW,CreutzWA}.

The main results of this paper establish these conjectures over global function fields, assuming a nonisotriviality hypothesis on the base change $A_\kbar$ of $A$ to an algebraic closure of $k$ and triviality of $\Sha(A)_\textup{div}$. To the best of our knowledge, this gives the first examples of subvarieties of abelian varieties (other than abelian varieties themselves) with $X(k)$ infinite where these conjectures are known to hold.
\newpage

\subsection{Statement of the main results}

\begin{Theorem}\label{thm:1}
	Let $A$ be an abelian variety over a global function field $k$ such that $A_{\kbar}$ has no nonzero isotrivial quotient. For every closed subvariety $X \subset A$, we have $$\overline{X(k)} = X(\A_k) \cap \overline{A(k)}\,.$$
\end{Theorem}

\begin{Theorem}\label{thm:2a}
	Let $A$ be an abelian variety over a global function field $k$ such that $A_{\kbar}$ has no nonzero isotrivial quotient. The following are equivalent.
	\begin{enumerate}
		\item\label{t2:1} $\Sha(A)_\textup{div}=0$,
		\item\label{t2:2} $\overline{A(k)} = A(\A_k)^{\Br}$,
		\item\label{t2:3} For every closed subvariety $X \subset A$ we have $\overline{X(k)} = X(\A_k)^{\Br}$.
	\end{enumerate}
\end{Theorem}

If one fixes the subvariety $X \subset A$, then a more precise result is possible. A subvariety $C \subset A$ is a \defi{coset} if there is an abelian subvariety $A' \subset A$ and a point $a \in A(\kbar)$ such that $C_{\kbar}$ is the translate of $A'_{\kbar}$ by the point $a \in A(\kbar)$. In this case we simply write $C = a + A'$. We emphasize that $A'$ and $C$ are defined over $k$, though $a$ need not be a $k$-rational point. A coset $C = a + A'$ is a torsor under the abelian variety $A'$ over $k$. Hence, a coset $C = a + A'$ with $C(\A_k)$ nonempty represents an element of $\Sha(A')$. Assuming $A_\kbar$ has no positive dimensional isotrivial quotient, the special subvarieties featuring in the Mordell-Lang conjecture mentioned above are cosets contained in $X$. The following gives an analogue of Mordell-Lang for the Brauer set.

\begin{Theorem}\label{thm:2}
	Let $A$ be an abelian variety over a global function field such that $A_{\kbar}$ has no nonzero isotrivial quotient. Let $X \subset A$ be a closed subvariety. There is a finite collection of cosets $C_i = a_i + A_i \subset X$, $i = 1,\dots,r$, such that 
	\[
		X(\A_k)^{\Br} = \bigcup_{i = 1}^r (C_i(\A_k)^{\Br})\,.
	\]
	Moreover, $\overline{X(k)} = X(\A_k)^{\Br}$ if and only if $\Sha(A_i)_\textup{div} = 0$ for all $i = 1,\dots,r$.
\end{Theorem}

If $X$ is \defi{coset free}, i.e., $X$ does not contain any coset of positive dimension, then the cosets $C_i$ appearing in Theorem~\ref{thm:2} are $k$-rational points and the theorem gives $X(k) = X(\A_k)^{\Br}$. Theorems~\ref{thm:1} and~\ref{thm:2} were proved for $X_{\kbar}$ coset free (which implies that $X$ is coset free, but not conversely) in \cite[Theorems B and D]{PoonenVoloch} under an additional hypothesis on the $p$-primary torsion of $A$. Theorem~\ref{thm:2} is proved for $X$ a nonisotrivial curve of genus at least $2$ (necessarily coset free) without the additional hypothesis on the $p$-primary torsion or the requirement that $A_\kbar$ has no positive dimensional isotrivial quotient in~\cite{CV3}. Partial results in the case of isotrivial curves may be found in~\cite{CV,CPV}.

In all of the results mentioned in the previous paragraph, the hypotheses imply that $X(L)$ is finite for every finite extension $L/k$, while there is no such requirement in the theorems stated above. An example of an $X$ covered by Theorem~\ref{thm:2} to which none of these previous results apply is given by the symmetric square of a nonhyperelliptic curve $Y$ of genus $3$ which admits a degree $2$ map to a nonisotrivial elliptic curve $E$. Then $X = \Sym^{2}(Y)$ contains a coset isomorphic to $E$. Suppose the Jacobian $A$ of $X$ splits up to isogeny as $A_\kbar \sim E_\kbar \times B$ for a simple nonisotrivial abelian surface $B/\kbar$. Theorem~\ref{thm:1} gives $\overline{X(k)} = X(\A_k) \cap \overline{A(k)}$. If, in addition, $\Sha(E)_{\textup{div}} = 0$, then Theorem~\ref{thm:2} gives $\overline{X(k)}= X(\A_k)^{\Br}$. Note that $X(k)$ is infinite if $E$ has positive rank.

%%%%%%%%%%%
\subsection{Adelic Mordell-Lang}
A key input to the proofs of the theorems above comes from the proofs of Mordell-Lang over function fields. These give not only a special subvariety of $X$ containing the rational points, but also a continuous version of this statement \cite[Theorem 6.4]{Hrushovski}. A similar statement was recently proved by Wisson~\cite{Wisson} using the `exceptional schemes' appearing in R\"ossler's proof of Mordell-Lang~\cite{RosslerML}. From these continuous versions one can fairly easily deduce the following theorem (See Section~\ref{sec:AML}). 

\begin{Theorem}\label{thm:AML1}
	Let $X \subset A$ be a closed subvariety of an abelian variety over a global function field $k$. Assume $A_{\kbar}$ has no positive dimensional isotrivial quotient. Then there exists a finite union of cosets $Y = \bigcup C_i$ contained in $X$ such that $X(\A_k) \cap \overline{A(k)} \subset Y(\A_k)$.
\end{Theorem}

With a bit more effort, we deduce a stronger version of this theorem in which $\overline{A(k)}$ is replaced by the profinite Selmer group $\Sel(A)$ (See Theorem~\ref{thm:AML}). This is a version of the `adelic Mordell-Lang conjecture' first stated in \cite{Stoll} over number fields for $X$ coset free (See Conjecture~\ref{conj:AML}). This was proved over function fields for $X$ coset free in \cite{PoonenVoloch} using results in~\cite{Hrushovski}. Our proof of Theorem~\ref{thm:AML} follows an approach similar to that taken in~\cite{PoonenVoloch}, but replaces the model theoretic input of~\cite{Hrushovski} there with results of \cite{PinkRossler,RosslerML,Wisson}. The conjecture remains wide open over number fields.

%%%%%%%%%%%
\subsection{Outline}

In Sections~\ref{sec:Sel} and~\ref{sec:cosets} we establish results that will allow us to deduce Theorem~\ref{thm:1} from Theorem~\ref{thm:AML1} and to deduce Theorem~\ref{thm:2} from the more general adelic Mordell-Lang conjecture just mentioned. The results in these sections also hold over number fields, at some points assuming the field is totally imaginary. 

One of the challenges to address is that the set $Y(\A_k) = (\bigcup C_i)(\A_k)$ of adelic points on the union featuring in Theorem~\ref{thm:AML1} is much larger than the union $\bigcup( C_i(\A_k))$ of the sets of adelic points on the components (the latter features in Theorem~\ref{thm:2}). One must show that any point in $Y(\A_k) \cap \overline{A(k)}$ is supported on one of the irreducible components. This was done for $X$ coset free (in which case $Y$ is finite) in \cite{Stoll,PoonenVoloch}, but it does not seem possible to deduce the general case from this. The key result we use is Proposition~\ref{lem:infords} below which generalizes \cite[Proposition 3.7]{Stoll} and \cite[Proposition 5.2]{PoonenVoloch}. This relies on a result of Serre in \cite{SerreII} concerning the image of Galois acting on torsion points of abelian varieties, and a number of additional results to handle the $p$-part in the function field case. 

The application of Proposition~\ref{lem:infords} to adelic points on finite unions of cosets is given in Section~\ref{sec:cosets}. In this section we assume that $k$ has no real primes. This restriction should not be necessary, but there are issues at real primes which would require additional arguments. See Section~\ref{sec:reals}.

In Section~\ref{sec:AML} we prove Theorem~\ref{thm:AML1} and use this to deduce Theorem~\ref{thm:1} before proving the more general adelic Mordell-Lang conjecture over function fields. The proof of the more general statement does not rely directly on Theorem~\ref{thm:AML1}. 

Proposition~\ref{lem:infords} is a statement about the profinite Selmer group $\Sel(A)$. In Section~\ref{sec:Br} we relate this to the Brauer set. It was previously known that $A(\A_k)^{\Br}_\bullet \subset \Sel(A)$ \cite{Skorobogatov,Stoll,PoonenVoloch}, with equality in the number field case by~\cite{CreutzBCyr}. Using recent results of Skorobogatov~\cite{Skorob} and Yang~\cite{Yang}, we prove this equality in the function field case (See Theorem~\ref{thm:notranscendental}). This allows us to prove the converse statements in Theorems~\ref{thm:2a} and~\ref{thm:2}, that the equality $\overline{X(k)} = X(\A_k)^{\Br}$ implies a finiteness result for the Tate-Shafarevich group.

 The theorems stated above are then proved in Section~\ref{sec:proofs}, where we also prove their analogues over totally imaginary number fields conditionally on the adelic Mordell-Lang conjecture.

%Finally, in Section~\ref{sec:isotrivial} we discuss some partial results in the case that $A_\kbar$ has a positive dimensional isotrivial quotient.

\subsection*{Acknowledgements}
		The author would like to thank Felipe Voloch for a number of helpful conversations, Alexei Skorobogatov for suggesting the use of~\cite{Skorob,Yang} in the proof of Theorem~\ref{thm:notranscendental}, James Milne for comments regarding Lemma~\ref{lem:Serre} and Damian R\"ossler and an anonymous referee for a number of helpful comments and corrections. This research was partially supported by Royal Society | Te Apārangi.

\section{Notation}

Throughout the paper $k$ is a global field, i.e., a number field or the function field of a geometrically integral curve over a finite field. We fix an algebraic closure $\kbar$ of $k$ and the separable closure $k^\sep$ of $k$ inside $\kbar$. The set of primes of $k$ is denoted by $\Omega_k$. The adele ring $\A_k$ is a topological ring defined as the restricted product $\prod_{v \in \Omega_k}(k_v,\calO_v)$, where $\calO_v$ is the ring of integers of the completion $k_v$ of $k$ at the prime $v$. 

A variety $Y$ over $k$ is a separated scheme of finite type over $k$. The set $Y(\A_k)$ is endowed with a topology from $\A_k$. When $Y$ is proper, we have that $Y(\A_k) = \prod_{v \in \Omega_k}Y(k_v)$ with the product topology of the $v$-adic topologies on $Y(k_v)$. 

For a topological space $T$ we use $T_\bullet$ to denote the set of connected components of $T$ with the induced topology. For a nonarchimedean prime $v \in \Omega_k$ we have $Y(k_v)_\bullet = Y(k_v)$. Thus, for $Y$ a proper variety over $k$ we have  $Y(\A_k)_\bullet = \prod_{v \in \Omega_k}Y(k_v)_\bullet = \prod_{v \nmid \infty}Y(k_v) \times \prod_{v \mid \infty}Y(k_v)_\bullet$. Note that when $k$ is a global function field we have $Y(\A_k)_\bullet = Y(\A_k)$.

\section{Selmer groups}\label{sec:Sel}
Throughout this section $A$ is an abelian variety over a global field $k$.

For any isogeny $f \in \End(A)$, the fppf cohomology of the exact sequence
\[
	0 \to A[f] \to A \stackrel{f} \to A \to 0
\]
gives rise to the following commutative diagram with exact rows.
\[\displaystyle
	\xymatrix{
	0 \ar[r]& A(k)/f(A(k)) \ar[r]\ar[d] & \HH^1(k,A[f]) \ar[d]\ar[r] \ar[dr]^{\alpha_{f}}& \HH^1(k,A)[f] \ar[r]\ar[d] & 0\\
	0 \ar[r]& \prod_{v \in \Omega_k} A(k_v)/f(A(k_v)) \ar[r]& \prod_{v \in \Omega_k} \HH^1(k_v,A[f]) \ar[r] & \prod_{v \in \Omega_k} \HH^1(k_v,A)[f] \ar[r] & 0
	}
\]
We define%\footnote{In \cite{Stoll,PoonenVoloch} the limits are denoted by $\widehat{\Sel}$ rather than $\Sel$ and also involve the complex primes.}
\begin{align*}
	\Sel^{f} &:= \ker(\alpha_{f})\, \\
	\Sel^{(f)} &:= \varprojlim_i \Sel^{f^i}\,,\text{ and}\\
	\Sel &:= \varprojlim_n \Sel^{n}\,,
\end{align*}
where the limits are with respect to the maps induced by $f^j: A[f^{i+j}] \to A[f^i]$ and by $m : A[mn] \to A[n]$ for integers $i,j,m,n \ge 1$. We will also write $\Sel(A)$, $\Sel^{n}(A)$, etc. if the abelian variety $A$ is not clear from the context. 

For any isogeny $f \in \End(A)$ and prime $v$ of $k$, exactness of the diagram above gives a map $\Sel^f \to A(k_v)/f(A(k_v))$. Passing to the limit one has $\Sel^{(f)} \to A(k_v)^{(f)} := \varprojlim A(k_v)/f^i(A(k_v))$. Similarly, there is a map $$\Sel \to \varprojlim_n \prod_{v\in \Omega_k} A(k_v)/nA(k_v) = \prod_{v \in \Omega_k} A(k_v)_{\bullet} = A(\A_k)_{\bullet}\,.$$

The goal of this section is to prove the following proposition.

	\begin{Proposition}\label{lem:infords}
		Suppose $P_1,\dots,P_r \in {\Sel}$ are all nonzero. Then there exists a positive density set of primes $S \subset \Omega_k$ such that for all $v \in S$ and all $j = 1,\dots,r$, the image of $P_j$ under ${\Sel} \to A(\A_k)_\bullet \to A(k_v)_\bullet$ is nonzero.
	\end{Proposition}
	
	In the case $r=1$, this recovers the result that the map $\Sel \to A(\A_k)_\bullet$ is injective, which is part of the Cassels dual exact sequence. See \cite[Proposition 4.3 and Remark 4.4]{PoonenVoloch} for a history of this result. In the number field case with $r=1$, this is also proved in~\cite[Proposition 3.7]{Stoll}. Our proof makes use of some of the ideas there.

\begin{Remark}
	Here is an example to show that a nonzero $P \in \Sel$ can have trivial image in $A(k_v)_\bullet$ for all $v$ in a set of density arbitrarily close to $1$. Let $n!_\ell := n!/\ell^{v_\ell(n!)}$ denote the prime to $\ell$ part of $n$ factorial. Suppose $Q \in A(k) \subset \Sel$ is a point of infinite order on an elliptic curve over a number field without complex multiplication. Consider the sequence $n!_\ell Q$, $n \in \mathbb{N}$. Since $\Sel$ is compact, we can extract a convergent subsequence converging to, say, $P \in \Sel$. If $v$ is a prime of good reduction with residue field $\F_v$ of characteristic prime to $\ell$ and such that the reduction of $Q$ has order prime to $\ell$, then the image of $P$ in $A(k_v)$ is $0$. By Chebotarev's theorem, the density of the set of primes such that $A(\F_v)$ contains a point of order $\ell$ is $O(1/\ell)$. On the other hand, the image of $P$ in $A(k_v)$ is nonzero at primes of residue characteristic $\ell$ and at primes where the reduction of $P$ has order divisible by $\ell$. So $P$ is not zero in $A(k_v)$ for all primes.
\end{Remark}

\subsection{The torsion subgroup}

For an abelian group $G$, let $G_\textup{tors}$ be its torsion subgroup.

\begin{Lemma}\label{lem:G}
	Let $f$ be an endomorphism of a finitely generated abelian group $G$ with finite kernel. The canonical map $G \to G^{(f)} = \varprojlim_i G/f^i$ induces an isomorphism $G[f^\infty] \simeq (G^{(f)})_\textup{tors}$.
\end{Lemma}

\begin{proof}
	Let $H = G/G_\textup{tors}$. Since $f$ has finite kernel it induces an injective map $H \to H$. The snake lemma applied to
	\[
		\xymatrix{
			0 \ar[r]& G_\textup{tors} \ar[d]^{f^i}\ar[r] & G \ar[d]^{f^i}\ar[r] & H \ar[d]^{f^i}\ar[r] & 0\\
			0 \ar[r]& G_\textup{tors} \ar[r] & G \ar[r] & H \ar[r] & 0
		}
	\]
	gives an exact sequence
	\[
		0 \to G_\textup{tors}/f^i \to G/f^i \to H/f^i \to 0\,.
	\]
	Taking limits and then torsion subgroups we obtain the exact sequence
	\begin{equation}\label{eq:Gs}
		0 \to (G_\textup{tors})^{(f)} \to (G^{(f)})_\textup{tors} \to (H^{(f)})_\textup{tors}\,.
	\end{equation}
	The group $H$ is isomorphic to $\Z^r$ for some $r$ and $f$ has finite kernel. Using Smith Normal Form for the endomorphisms $f^i : H \to H$ we get that $H^{(f)}\otimes \Z_\ell$ is isomorphic to a product of at most $r$ copies of $\Z_\ell$, for every prime $\ell$. It follows that $H^{(f)}$ is torsion free. Since $G$ is finitely generated, for sufficiently large $i$ we have that $f^i$ induces the $0$ map on $G[f^\infty]$ and $f^i$ induces an automorphism of the quotient $G_\textup{tors}/G[f^\infty]$. So, for sufficiently large $i$, we have $G_\textup{tors}/f^i \simeq G[f^\infty]$. It follows that $\left(G_\textup{tors}\right)^{(f)} = G[f^\infty]$. The result now follows from~\eqref{eq:Gs}.
\end{proof}

\begin{Lemma}\label{lem:tors}
	Let $f \in \End(A)$ be an isogeny. The natural maps $A(k) \to A(k)^{(f)} \to \Sel^{(f)}$ induce isomorphisms $A(k)[f^\infty] \simeq A(k)^{(f)}_\textup{tors} \simeq \Sel^{(f)}_\textup{tors}$.
\end{Lemma}

\begin{proof}
	The first isomorphism follows from Lemma~\ref{lem:G}, which applies since $A(k)$ is finitely generated by the Mordell-Weil theorem. For the second, taking limits of the exact sequences $0 \to A(k)/f^i(A(k)) \to \Sel^{f^i} \to \Sha(A)[f^i] \to 0$ gives an exact sequence $0 \to A(k)^{(f)} \to \Sel^{(f)} \to \varprojlim_i \Sha(A)[f^i]\,.$ The last term is torsion free, so the torsion subgroups of the first two terms are isomorphic.
	\end{proof}

\subsection{The pro-$p$ Selmer group}
We collect here some results relevant for the case $p > 0$.

\begin{Definition}\label{def:psplit}
Suppose $k$ has characteristic $p > 0$. We say that $A$ is $p$-\defi{split} if there exist endomorphisms $f,g \in \End(A)$ and an integer $r\ge 0$ such that 
\begin{enumerate}
	\item $p^r = fg = gf$ in $\End(A)$,
	\item $f$ is \'etale,
	\item $A(k^\sep)[p^\infty] = A(k^\sep)[f^\infty]$, and
	\item $A(k^\sep)[g^\infty]=0$.
\end{enumerate}
\end{Definition}

\begin{Lemma}\label{lem:Selpsplit}
If $A$ is $p$-split, then $\Sel^{(p)} \simeq \Sel^{(f)} \times \Sel^{(g)}$.
\end{Lemma}

\begin{proof}
	Suppose $A$ is $p$-split with endomorphisms $f,g$ as in the definition and let $i \ge 1$. Then $A[f^i]$ and $A[g^i]$ have trivial intersection, since $f^i$ is \'etale and $A(k^\sep)[g^i] = 0$. Since $p^{ri} = f^ig^i = g^if^i$, it follows that $A[p^{ri}] \simeq A[f^i]\times A[g^i]$ as group schemes. Consequently, $\Sel^{p^{ri}} \simeq \Sel^{f^i} \times \Sel^{g^i}$. Moreover, this splitting is compatible with the maps $p^r_* : \Sel^{p^{r(i+1)}} \to \Sel^{p^{ri}}$. Since the groups $\Sel^{p^{ri}}$ are cofinal among the $\Sel^{p^i}$, we have $\Sel^{(p)} \simeq \Sel^{(f)} \times \Sel^{(g)}$.
\end{proof}

The following lemma specializes to \cite[Proposition 5.2]{PoonenVoloch} in the case that $A(k^\sep)[p] = 0$ (which implies that $A$ is $p$-split with $f = 1$ and $g = p$). A proof of the more general statement given here can be found in the proof of \cite[Lemma 3.2]{CV3}.

\begin{Lemma}\label{lem:pcSel}
	Suppose that $A$ is $p$-split. For any prime $v$ of $k$, the canonical map $\Sel^{(g)} \to A(k_v)^{(g)} = \varprojlim_i A(k_v)/g^i(A(k_v))$ is injective.
\end{Lemma}

While not all abelian varieties are $p$-split, the following result due to R\"ossler shows that all are isogenous over some finite extension to a $p$-split abelian variety.

\begin{Lemma}\label{lem:Rossler1}
	Suppose $k$ has characteristic $p > 0$. There exists a finite separable extension $L/k$, a $p$-split abelian variety $B/L$, and an \'etale isogeny $A_L \to B$. 
\end{Lemma}

\begin{proof}
	See \cite[Appendix by R\"ossler]{CV3}. Note that after base change to some finite separable extension $A$ will satisfy the conditions of \cite[Proposition A.1]{CV3}.
\end{proof}

\subsection{A result of Serre}

The following result due to Serre \cite[p. 734, Corollaire]{SerreII} will be used in the proof of Proposition~\ref{lem:infords}. Our proof of Proposition~\ref{lem:infords} given below uses some ideas from \cite[Section 3]{Stoll}, which are based on stronger `image of Galois' results of Bogomolov and Serre. Those results are, however, not valid in the global function field case \cite{Zarhin}.

\begin{Lemma}\label{lem:Serre}
	Let $\ell$ be a prime number and let $T_\ell = \varprojlim_i A(k^\sep)[\ell^i]$. Let $G_\ell \subset \Aut(T_\ell)$ denote the image of the representation describing the action of $\Gal(k^\sep/k)$ on $T_\ell$. Then $\HH^1(G_\ell,T_\ell)$ is finite. 
\end{Lemma}

\begin{proof}
	When $\ell$ is not equal to the characteristic of $k$, this is~\cite[p. 734, Corollaire and Remarques 2)]{SerreII}. We repeat part of the argument given there, explaining how to extend to the case when $\ell$ is equal to the characteristic.
	
	First note that $T_\ell$ is a free $\Z_\ell$-module of finite rank and that $G_\ell$ is a compact $\ell$-adic Lie subgroup of $\Aut(T_\ell)$. Let $\mathfrak{g}_\ell$ be the Lie algebra of $G_\ell$ and let $\log : G_\ell \to \mathfrak{g}_\ell$ be the $\ell$-adic logarithm. Let $V_\ell = T_\ell \otimes \Q_\ell$. As in the proof of \cite[p. 734, Corollaire]{SerreII}, finiteness of $\HH^1(G_\ell,T_\ell)$ will follow from showing that $\HH^1(\mathfrak{g}_\ell,V_\ell) = 0$. For this, \cite[Th\'eor\`eme 1]{SerreII} shows that it is enough to exhibit an element $x \in \mathfrak{g}_\ell$ whose spectrum satisfies condition $(P_N)$ on \cite[p.732]{SerreII}. Such an element can be obtained as the logarithm of a lift of Frobenius at a prime of good reduction where, in the case $\ell$ is equal to the characteristic of $k$, the $\ell$-rank does not drop. Condition $(P_N)$ for $x$ follows from the fact that the spectrum of the lift of Frobenius consists of Weil $q$-numbers, where $q$ is the cardinality of the residue field. Further details are provided below.

	Let $v$ be a prime of good reduction for $A$ with residue field $\F_v$ such that $A(\kbar)[\ell]$ and $A_v(\overline{\F}_v)[\ell]$ have the same dimension as vector spaces over $\F_\ell$. Here $A_v$ denotes the reduction of $A$ at the prime $v$. Let $F \in \Gal(\F_v)$ be the map $F : a \mapsto a^q$, where $q = |\F_v|$. Let $\phi \in \End(A_v)$ be the endomorphism induced by $F$, and let $F'$ be a lift of $F$ to $\Gal(k)$. By our assumption on $v$, the reduction map induces an isomorphism $\varprojlim A(\kbar)[\ell^n] \to \varprojlim A_v(\overline{\F}_v)[\ell^n]$ and an injective homomorphism $T_\ell = \varprojlim A(k^\sep)[\ell^n] \to \varprojlim A_v(\overline{\F}_v)[\ell^n]$ that is equivariant with respect to $F'$ and $F$. It follows that the characteristic polynomial $\chi_{F'} \in \Z_\ell[t]$ of $F'$ acting on $T_\ell$ divides the characteristic polynomial $\chi_{F} \in \Z_\ell[t]$ of $F$ acting on $\varprojlim A_v(\overline{\F}_v)[\ell^n]$. Let $\chi_\phi \in \Z[t]$ denote the characteristic polynomial of $\phi$ as an element in $\End(A_v)$. Then $\chi_\phi(\phi)$ is the $0$ map on all of $A_v(\overline{\F}_v)$, so $\chi_\phi(\phi)$ also acts as $0$ on $\varprojlim A_v(\overline{\F}_v)[\ell^n]$. It follows that the roots of $\chi_{F'}$ and $\chi_{F}$ are roots of $\chi_{\phi}$.
	
	Let $x = \log(F')$ in $\mathfrak{g}_\ell$. As in the proof of \cite[Lemme 3]{SerreII}, the spectrum of $x$ consists of the logs of the roots of $\chi_{F'}$. As is well known, the roots of $\chi_\phi$ are Weil $q$-numbers, hence the same is true for the roots of $\chi_{F'}$. From this one deduces as in~\cite{SerreII} that the spectrum of $x$ satisfies his condition $(P_N)$.
	\end{proof}

\subsection{The proof of Proposition~\ref{lem:infords}}
	
	\begin{proof}
		First note that if $0 \ne P_j \in \Sel_\textup{tors} = A(k)_\textup{tors}$, then the image of $P_j$ in $A(k_v)$ is nonzero for all but finitely many primes $v$. Hence we can assume all of the $P_j$ are of infinite order.
		
		For this paragraph assume that $k$ has characteristic $p > 0$. By Lemma~\ref{lem:Rossler1} there is a finite separable extension $L/k$ and an isogeny $A_L \to B$ such that $B$ is $p$-split. Let $f,g \in \End(B)$ be the endomorphisms as in Definition~\ref{def:psplit} with $fg = gf = p^r$. The induced map ${\Sel}(A) \to {\Sel}(A_L) \to {\Sel}(B)$ has finite kernel. Hence, the images $P'_j$ of the $P_j$ in $\Sel(B)$ all have infinite order. If there is a prime $w$ of $L$ such that all of these $P_j'$ have nonzero image in $B(L_w)$, then the corresponding $P_j$ all have nonzero image in $A(k_v)$ for the prime $v$ below $w$. So, replacing $A$ with $B$ and the $P_j$ with the $P_j'$, we can assume that $A$ is $p$-split. Then, by Lemma~\ref{lem:Selpsplit}, the profinite Selmer group splits as
		$$\Sel = \Sel^{(f)} \times \Sel^{(g)} \times \prod_{\ell \ne p}\Sel^{(\ell)}\,.$$
For every prime $v$ of $k$, the map $\Sel^{(g)} \to A(k_v)^{(g)}$ is injective by Lemma~\ref{lem:pcSel}. So, if the image of $P_j$ in $\Sel^{(g)}$ is nonzero, then the image of $P_j$ in $A(k_v)$ is nonzero for every prime $v$. We may therefore assume that all of the $P_j$ have trivial image in $\Sel^{(g)}$ (by considering instead the subset of the $P_j$ for which this is the case). Hence, we assume that the image of $P_j$ in $\Sel^{(f)}\times \prod_{\ell \ne p}\Sel^{(\ell)}$ has infinite order, for all $j$.

For consistency of notation in the case that $p = 0$, set $f = 1 \in \End(A)$. Then we have $\Sel = \prod_{\ell \ne p}\Sel^{(\ell)} = \Sel^{(f)}\times \prod_{\ell \ne p}\Sel^{(\ell)}$ and the assumption that the images of the $P_j$ in this group have infinite order.

		Let $M = |A(k)_\textup{tors}|$, which is a positive integer by the Mordell-Weil theorem. Let $Q_j = MP_j$. For any \'etale isogeny $h \in \End(A)$, Lemma~\ref{lem:tors} gives $\Sel^{(h)}_\textup{tors} = A(k)^{(h)}_\textup{tors} \simeq A(k)[h^\infty] \subset A(k)_\textup{tors}$. So, for each $j$, the image of $Q_j$ in $\Sel^{(h)}$ is either trivial or has infinite order. Since the $Q_j$ all have infinite order in $\Sel^{(f)} \times \prod_{\ell \ne p}\Sel^{(\ell)}$, there is some \'etale isogeny $m = m'f$ with $m'$ an integer not divisible by $p$ such that the $Q_j$ all have infinite order in $\Sel^{(m)}$. 

		Since limits are left exact, the definition of the Selmer groups gives an injective map $\Sel^{(m)} = \varprojlim \Sel^{m^i} \hookrightarrow \varprojlim \HH^1(k,A[m^i])$. By the universal property of limits there is a canonical map $\HH^1(k,T_m(A)) = \HH^1(k,\varprojlim A[m^i]) \to \varprojlim \HH^1(k,A[m^i])$. Since the groups $\HH^0(k,A[m^i]) = A(k)[m^i]$ are finite, this canonical map is an isomorphism by \cite[Corollary 2.7.6]{CoNF}. Hence, we have an injective map $\Sel^{(m)} \hookrightarrow \HH^1(k,T_m(A))$.
		
		For $i \ge 1$, let $k_i = k(A[m^i])$ and let $k_\infty = k(A[m^\infty]) = \bigcup k_i$. The restriction maps, together with the canonical projections $T_m(A) \to A[m^i]$, then give the following commutative diagram of Galois cohomology groups.
		 \begin{equation}\label{eq:kinf}
		 	\xymatrix{
			\Sel^{(m)} \ar@{^{(}->}[r] &\HH^1(k,T_m(A)) \ar[rrrr]^{\res_{k_\infty/k}} \ar[d] &&&& \HH^1(k_\infty,T_m(A)) \ar[d] \\
			&\HH^1(k,A[m^i]) \ar[rr]^{\res_{k_i/k}} && \HH^1(k_i,A[m^i]) \ar[rr]^{\res_{k_\infty/k_i}} && \HH^1(k_\infty,A[m^i])
			}
		 \end{equation}
Recall from Definition~\ref{def:psplit} that $A(k^\sep)[f^\infty] = A(k^\sep)[p^\infty]$. It follows that \[T_m(A) := \varprojlim A(k^\sep)[m^i] = T_p(A) \times \prod_{\ell \mid m'} T_\ell(A) = \prod_{\ell \mid \deg(m)}T_\ell(A)\,,\] where $T_\ell(A) = \varprojlim A(k^\sep)[\ell^n]$. Hence, the kernel of the restriction map in the top row of the diagram above is finite by Lemma~\ref{lem:Serre}. So the images of the $Q_j$ in $\HH^1(k_\infty,T_m(A))$ have infinite order. Since $\HH^1(k_\infty,T_m(A)) = \varprojlim \HH^1(k_\infty,A[m^i])$, this implies that we can choose $N$ so that for all $i \ge N$ and all $j = 1,\dots, r$, the order of the image of $Q_j$ in $\HH^1(k_\infty,A[m^i])$ is greater than $r$. From the diagram~\eqref{eq:kinf} it follows that the images of the $Q_j$ in $\HH^1(k_N,A[m^N]) = \Hom(\Gal(k_N),A[m^N])$ all have order greater than $r$.
		 
		 Let $L_j$ denote the fixed field of the kernel of the homomorphism corresponding to the image of $Q_j$ in $\Hom(\Gal(k_N),A[m^N])$.		 The extensions $L_j/k_N$ have degree strictly greater than $r$. So by Chebotarev's theorem, the density of primes of $k_N$ that split completely in $L_j$ is less than $1/r$. It follows that there is a positive density set of primes of $k_N$ which do not split completely in $L_j$ for any $j = 1,\dots,r$. Note that a prime $w$ of $k_N$ splits completely in $L_j$ if and only if $Q_j$ has trivial image in $A((k_N)_w)/m^N(A((k_N)_w)) \subset \HH^1((k_N)_w,A[m^N])$. Hence there is a positive density set of primes $v \in \Omega_{k_N}$ such that the $v$-adic component of $Q_j$ is nontrivial for all $j = 1,\dots, r$. If $v_0$ denotes the prime of $k$ below such a $v$, then the $v_0$-adic components of the $Q_j$ are also all nonzero. Since $Q_j = MP_j$, the same is true of the $P_j$. This completes the proof.
	\end{proof}

\section{Adelic intersections for unions of cosets}\label{sec:cosets}

In this section $k$ is either a global function field or a totally imaginary number field. Equivalently, $k$ is a global field with no real primes.

For $X$ a closed subvariety of the abelian variety $A$ over $k$ we define $X(\A_k)_\circ = \prod_{v \in \Omega^\circ_k}X(k_v)$ to be the product over the set $\Omega_k^\circ$ of nonarchimedean primes. Since all the archimedean primes are assumed to be complex and $A$ is geometrically connected, the canonical map $A(\A_k)_\bullet \to A(\A_k)_\circ$ is an isomorphism. Moreover, the map $X(\A_k)_\circ \to A(\A_k)_\circ$ induced by the inclusion $X(\A_k) \subset A(\A_k)$ is injective.

The results of this section concern the images in $A(\A_k)_\circ$ of various subsets of $A(\A_k)$ and $A(\A_k)_\bullet$. We will simplify the notation by writing $X(\A_k), \overline{X(k)}, A(\A_k), \Sel(A)$, etc. to also denote their images in $A(\A_k)_\circ$.

	\begin{Lemma}\label{lem:Yi}
		Let $Y = \bigcup_{j = 1}^r C_j$ be a finite union of cosets $C_j \subset A$. Then $$Y(\A_k) \cap \Sel(A) \subset \bigcup_{j = 1}^r \left(C_j(\A_k)\right).$$
	\end{Lemma}
	
	\begin{proof}
		Let $Q \in Y(\A_k) \cap \Sel(A)$. Suppose $C_j = a_j + A_j$ with $a_j \in A(\kbar)$ and subabelian varieties $A_j \subset A$. Consider the quotient maps $\phi_j : A \to B_j := A/A_j$, and let $B = \prod B_j$ with the canonical maps $\psi_j : B_j \to B$. Then $\psi_j(\phi_j(C_j))$ is a point $x_j \in B(k)$. Let $Q_j = \psi_j(\phi_j(Q)) \in B(\A_k)$. Then $Q_j \in \Sel(B)$. Let $P_j = Q_j - x_j \in {\Sel}(B)$. The assumption that $Q \in Y(\A_k) = (\bigcup C_j)(\A_k)$ implies that for every prime $v \in \Omega_k$, there is some $j\in \{1,\dots,r\}$ such that the image of $P_j$ under the projection $B(\A_k) \to B(k_v)$ is $0$. By Proposition~\ref{lem:infords}, this implies that there exists some $j \in \{1,\dots,r\}$ such that $P_j = 0$ in $\Sel(B)$. For this $j$ we have $Q_j = x_j$, and it follows that the $v$-adic component of $Q$ lies on $C_j$ for all nonarchimedean primes $v$. This means that (the image of) $Q$ lies in (the image of) $C_j(\A_k)$ (in $A(\A_k)_\circ$).
	\end{proof}

\begin{Lemma}\label{lem:Y'}
	Let $K/k$ be a finite Galois extension contained in $k^\sep$. Suppose $Y \subset A_K$ is a subvariety such that the reduced subschemes of the irreducible components of $Y_{\kbar}$ are cosets in $A_{\kbar}$. Then there exists a subvariety $Y' \subset A$ such that
	\begin{enumerate}
		\item $Y'$ is a finite union of cosets in $A$,
		\item $Y'_K \subset Y$ as subvarieties of $A_K$, and
		\item The intersection of $Y(\A_K)$ with the image of $\Sel(A)$ under $A(\A_k) \to A_K(\A_K)$ is contained in the image of $Y'(\A_k) \to Y'_K(\A_K)$.
	\end{enumerate}
\end{Lemma}

\begin{proof}
	Let $Y_0$ be the reduced subscheme of the Zariski closure of $Y(k^\sep)$. Then $Y_0$ is reduced and $Y_0(k^\sep)$ is dense in $Y_0$, so $Y_0$ is geometrically reduced. Let $w$ be a nonarchimedean prime of $K$ and let $K_w^h$ denote the henselization and $K_w$ denote the completion. By Greenberg's approximation theorem \cite{Greenberg}, we have that $Y(K_w^h)$ is dense in $Y(K_w)$. By~\cite[Lemma 3.1]{PoonenVoloch}, $K_w^h/K$ is separable, so $Y(K_w^h) = Y_0(K_w^h)$. We conclude that $Y(K_w) = Y_0(K_w)$ and $Y(\A_K) = Y_0(\A_K)$. By assumption $(Y_0)_\kbar = \bigcup C_j$ with $C_j = a_j + A_j$ for some $a_j \in A(\kbar)$ and abelian subvarieties $A_j \subset A_{\kbar}$. Every abelian subvariety of $A_{\kbar}$ can be defined over $k^\sep$ (See \cite[Corollary 6]{Liu}). Hence, we may assume that the $a_j$ are in $A(k^\sep)$, and that $A_j$ and $C_j$ are subvarieties of $A_{k^\sep}$.
	
	Let $L/k$ be a finite Galois extension containing $K$ such that, for all $j$, $A_j$ is defined over $L$ and $a_j \in A(L)$. For each $j$, let $Z_j = \bigcap_{\sigma \in \Gal(L/k)} \sigma(C_j)$. Then $Z_j$ is defined over $k$ and the irreducible components of $(Z_j)_{k^\sep}$ are translates of the identity component $B_j$ of $\bigcap_{\sigma}\sigma(A_j)$, which is an abelian subvariety of $A$. Let $Y'$ be the union of the irreducible components of $\bigcup_j Z_j$ that are geometrically irreducible. Then $Y'$ is a finite union of cosets in $A$ and it is clear that $Y'_K \subset Y$.
	
	Let $y \in Y(\A_K) \cap \Sel(A)$, the intersection taking place in $A(\A_K)$. We need to show $y \in Y'(\A_k)$. The image of $\Sel(A) \to A_L(\A_L)$ is contained in $\Sel(A_L)$. So by Lemma~\ref{lem:Yi} applied to $Y_L$ we have that $y \in C_j(\A_L)$ for some $j$. Since $y \in A(\A_k) = A(\A_L)^{\Gal(L/k)}$ (the equality holds by~\cite[Lemma 3.2]{PoonenVoloch}), we must have that $y \in Z_j(\A_k) \cap \Sel(A)$. Repeating the argument with $Z_j$ in place of $Y$ we find that $y$ lies on the intersection over a Galois orbit of irreducible components of $(Z_j)_{k^\sep}$. Since all of the irreducible components of $(Z_j)_{k^\sep}$ are translates of the same abelian subvariety $B_j$ (which is defined over $k$), the intersection over any Galois orbit of size greater than $1$ is empty. Hence $y$ is an adelic point on some irreducible component of $Z_j$ which is geometrically irreducible. So $y \in Y'(\A_k)$ as required.
\end{proof}

The preceding lemmas allow us to recover the following result proved by Stoll in the number field case \cite[Prop. 3.6]{Stoll} and by Poonen-Voloch in the function field case  \cite[Prop. 5.3]{PoonenVoloch}. The latter required an extra hypothesis on the $p$-primary torsion subgroup of $A$. This extra hypothesis can be removed using the result of R\"ossler, Lemma~\ref{lem:Rossler1}. See \cite[Prop. 3.1]{CV3}.
\begin{Corollary}\label{cor:zeroD}
	If $Y$ is a finite subscheme of $A$ then $Y(\A_k) \cap \Sel(A)= Y(k)$.
\end{Corollary}

\begin{proof}
	In this case the reduced subscheme of $Y_{\kbar}$ is a finite union of cosets of the trivial subgroup of $A_{\kbar}$. The $Y' \subset Y$ supplied by Lemma~\ref{lem:Y'} is $Y' = \bigcup_{y_j \in Y(k)} y_j$. By Lemma~\ref{lem:Yi} we have \[ Y(\A_k) \cap \Sel(A) \subset Y'(\A_k) \cap \Sel(A) \subset \bigcup_{y_j \in Y(k)} (y_j(\A_k)) = Y(k).\]
\end{proof}

The following lemma will be used in the proof of Theorem~\ref{thm:1}. We will prove an analogue of this with $\Sel(A)$ in place of $\overline{A(k)}$ in Lemma~\ref{lem:descenttheory}.

\begin{Lemma}\label{lem:T}
	Let $C \subset A$ be a coset. Then $C(\A_k) \cap \overline{A(k)} = \overline{C(k)}$.
\end{Lemma}

\begin{proof}
	Suppose $C = a + A'$ is a coset of $A' \subset A$. By Poincar\'e reducibility, there is an abelian subvariety $A'' \subset A$ such that $A' \cap A''$ is finite and the sum map $\phi : A' \times A'' \to A$ is an isogeny. To ease notation let $B = A' \times A''$. For some integer $n \ge 2$ we have that multiplication by $n$ on $A$ factors as $\phi \circ \psi$ for some isogeny $\psi : A \to B$. 
	
	The containment $\overline{C(k)} \subset C(\A_k) \cap \overline{A(k)}$ is obvious. For the reverse containment, suppose $c \in C(\A_k) \cap \overline{A(k)}$. By \cite[Proposition 4.1]{PoonenVoloch} which is based on results of Serre \cite{SerreII} and Milne \cite{MilneCongruence}, we have that $\overline{A(k)} \simeq A(k)\otimes\hat{\Z}$ as topological groups. Since $A(k)$ is finitely generated, the subgroup $n(A(k)\otimes\hat{\Z}) \subset A(k)\otimes\hat{\Z}$ is of finite index and, hence, open. Since $A(k)$ is dense in $A(k) \otimes \hat{\Z}$, each coset of $n(A(k)\otimes\hat{\Z})$ contains an element of $A(k)$. It follows that $c = na + P$ for some $a \in \overline{A(k)}$ and $P \in A(k)$. 
	
	Let $b = \psi(a) \in \overline{B(k)}$, so that $c = \phi(b) + P$. Let $D \subset B$ be the pullback of $C \subset A$ under the morphism $B \to A$ sending a point $z$ to $\phi(z) + P$. Then $b \in D(\A_k) \cap \overline{B(k)}$. Since $c = \phi(b) + P$, it will be enough to show that $b \in \overline{D(k)}$.
	
	Let $Z \subset A''$ be the image of $D$ under the quotient $B = A'\times A'' \to A''$ and let $z \in Z(\A_k)$ be the image of $b$. The reduced subschemes of the irreducible components of $D_{\kbar}$ are translates of $A'_{\kbar}$, so $Z$ is finite. Then $Z(\A_k) \cap \overline{A''(k)} = Z(k)$ by Corollary \ref{cor:zeroD}. Thus $z\in A''(k)$. It follows that
	\[ b \in (A'\times\{z\})(\A_k) \cap \overline{B(k)} = \overline{A'(k)} \times \{z\} \subset \overline{D(k)}\,. \]
\end{proof}

\subsection{Remarks in the case that $k$ is a number field with real primes}\label{sec:reals}

	We expect the results of this section also hold for number fields with real primes, if one considers all of the intersections and containments as taking place in $A(\A_k)_\bullet$. However, additional arguments would be required to prove this. We outline here the points at which our proofs break down in the presence of real primes.

	\subsubsection*{Lemma~\ref{lem:Yi}} In the proof we showed that the adelic point $Q \in Y(\A_k) \cap \Sel(A)$ has all of its nonarchimedean $v$-adic components lying on some coset $C \subset A$. From this we wish to conclude that there is some $(y_v) \in C(\A_k)$ such that $Q$ and $(y_v)$ have the same image in $A(\A_k)_\bullet$. That is, for each archimedean prime $w$, we wish to find $y_w \in C(k_w)$ such that $y_w$ and $Q$ have the same image in $A(k_w)_\bullet$. For complex primes there is no issue because $C(\C) \ne \emptyset$ and $A(\C)_\bullet$ is a singleton. For real primes, however, we do not even know that $C(\R) \ne \emptyset$.
	
	\subsubsection*{Lemma~\ref{lem:Y'}} In the proof we used that $A(\A_L)^{\Gal(L/k)} = A(\A_k)$ for a finite Galois extension $L/k$ to show that the adelic point $y \in Y(\A_L) \cap \Sel(A_L)$ in $A(\A_L)$ lies in the image of the map $Y'(\A_k) \to Y(\A_L)$. When there are real primes the map $A(\A_k)_\bullet \to A(\A_L)_\bullet$ need not be injective, so this argument breaks down. Similar to the situation in the previous paragraph, some additional argument is required to show that $Y(\A_L) \cap \Sel(A)$ being nonempty implies that $Y'(\R)$ is nonempty.

\section{Adelic Mordell-Lang}\label{sec:AML}

The following was suggested in~\cite[Question 3.12]{Stoll} for coset free $X$ over a number field.

\begin{Conjecture}\label{conj:AML}
	Let $X \subset A$ be a closed subvariety of an abelian variety over a global field $k$. Assume $A_{\kbar}$ has no positive dimensional isotrivial quotient. Then there exists a finite union of cosets $Y = \bigcup C_i$ contained in $X$ such that
	\[
		X(\A_k)_{\bullet_A} \cap \Sel(A) \subset Y(\A_k)_{\bullet_A}\,,
	\]
	where $T_{\bullet_A}$ denotes the image of a subset $T \subset A(\A_k)$ in $A(\A_k)_\bullet$.
\end{Conjecture}

In Theorem~\ref{thm:AML} below we prove this conjecture in the case that $k$ is a global function field, in which case the conclusion simplifies to the statement that $X(\A_k) \cap \Sel(A) \subset Y(\A_k)$. The proof we give is based on ideas of Wisson~\cite{Wisson,WissonThesis} and R\"ossler's proof of the Mordell-Lang conjecture \cite{RosslerML} over global function fields which depends on the Manin-Mumford conjecture proved by Pink and R\"ossler \cite{PinkRossler}. The precise results we use are in Lemma~\ref{lem:Wisson} and Lemma~\ref{lem:Rossler} below.

Theorem~\ref{thm:AML} was proved in the coset free case in \cite[Lemma 3.18]{PoonenVoloch} by establishing a more uniform version of \cite[Lemma 6.2]{Hrushovski}. We expect that a similar approach could be carried out in the general case (i.e., without the coset free hypothesis), but have not attempted this. The main result of~\cite{Wisson} appeared much earlier in~\cite[Theorem 6.4]{Hrushovski} and we expect that this could be similarly adapted to prove Theorem~\ref{thm:AML}.

\subsection{A weaker version of the conjecture}
While it is not needed for the proof of Theorem~\ref{thm:AML}, we first explain how to deduce the weaker version of the conjecture with $\overline{A(k)}$ in place of $\Sel(A)$ directly from the main result of Wisson \cite{Wisson} or~\cite[Theorem 6.4]{Hrushovski}. This is Theorem~\ref{thm:AML1} of the introduction, and together with the results of the earlier sections is enough to prove Theorem~\ref{thm:1}.

\begin{Theorem}\label{thm:AMLmw}
Let $A$ be an abelian variety over a global function field $k$ such that $A_{\kbar}$ has no positive dimensional isotrivial quotient. Let $X \subset A$ be a closed subvariety. Then there is a finite union of cosets $Y \subset X$ such that $X(\A_k) \cap \overline{A(k)} \subset Y(\A_k)$.
\end{Theorem}

\begin{proof}
Suppose $k$ is the function field of a curve over a finite field $\F$. Let $k' = k\overline{\F}$ be the base change of $k$ to $\overline{\F}$. The main result of \cite{Wisson} asserts the existence of a subvariety $W \subset X_{k'}$ with the following properties: 
\begin{enumerate}
	\item $W_{\kbar}$ is a finite union of cosets in $A_{\kbar}$, and
	\item for every discrete valuation $w$ of $k'$ and metric $d_w$ on $A(k'_w)$ inducing the $w$-adic topology, there exists a positive real number $c_w$ such that for every $P$ in the finitely generated subgroup $A(k) \subset A(k')$ we have $d_w(W,P) \le c_w\cdot d_w(X_{k'},P)$. 
\end{enumerate}
Alternatively, \cite[Theorem 6.4]{Hrushovski} gives the existence of such a $W$, after possibly passing to some extension of $k'$.

For some finite Galois extension $L/k$ contained in $k'$ there is a model $Y$ of $W$ over $L$, i.e. a subvariety $Y \subset X_L$ such that the base change of $Y$ to $k'$ is equal to $W$ as a subvariety of $X_{k'}$. From (2) above it follows that $X_{k'}(\A_{k'}) \cap \overline{A(k)} \subset W(\A_{k'})$, where the topological closure is taken in $A_{k'}(\A_{k'})$. The inclusions $A(\A_k) \subset A_L(\A_L) \subset A_{k'}(\A_{k'})$ are homeomorphisms onto their images, which are closed subsets. So the closure of $A(k)$ in $A_{k'}(\A_{k'})$ is equal to the image of its closure in $A(\A_k)$. It follows that
\[
	X(\A_k) \cap \overline{A(k)} \subset W(\A_{k'}) \cap \overline{A(k)} = Y_{k'}(\A_{k'}) \cap \overline{A(k)} \subset Y(\A_L) \cap \overline{A(k)}.
\]
Lemma~\ref{lem:Y'} gives a finite union of cosets $Y' \subset X$ such that $Y(\A_L) \cap \overline{A(k)} \subset Y'(\A_k)$ inside $A(\A_L)$. So $Y' \subset X$ is a finite union of cosets in $A$ such that
$
	X(\A_k) \cap \overline{A(k)} \subset Y'(\A_k)\,.
$
\end{proof}

\begin{proof}[Proof of Theorem~\ref{thm:1}]
	By Theorem~\ref{thm:AMLmw} there is a finite union of cosets $Y = \bigcup_i C_i \subset X$ such that $X(\A_k) \cap \overline{A(k)} \subset Y(\A_k)$. Lemma~\ref{lem:Yi} and Lemma~\ref{lem:T} give
	\[
		X(\A_k) \cap \overline{A(k)} \subset Y(\A_k) \cap \overline{A(k)} \,\stackrel{\ref{lem:Yi}}{\subset}\, \bigcup_{i}\left( C_i(\A_k) \cap \overline{A(k)} \right) \,\stackrel{\ref{lem:T}}{=}\, \bigcup_i \overline{C_i(k)} \subset \overline{X(k)}\,.
	\]
	Clearly, $\overline{X(k)} \subset X(\A_k) \cap \overline{A(k)}$, so these two sets are equal.
	
\end{proof}

\subsection{R\"ossler's exceptional schemes}

Let us fix some notation in effect for the remainder of this section. Suppose $X$ is a closed subvariety of an abelian variety $A$ over $k = \F(U)$, where $U$ is a curve over the finite field $\F$ of characteristic $p$. We assume that $U$ is small enough so that $A$ spreads out to an abelian scheme $\calA/U$ and $X$ spreads out to $\calX/U$ which is a closed subscheme of $\calA$. Let $k' = k \Fbar$, $U' = U \times_\F \overline{\F}$, $X' = X \times_k k'$, $\calA' = \calA \times_UU'$ and $\calX' = \calX \times_U U'$.

R\"ossler uses jet schemes in~\cite{RosslerML} to construct the `exceptional schemes' $\Exc^n(\calA,\calX)/U'$ for each $n \ge 1$ (See~\cite[Section 3A]{RosslerML}). These were used by Wisson to construct the variety $W \subset X_{k'}$ in the proof of Theorem~\ref{thm:AMLmw} above. The exceptional schemes are closed subschemes of $\calX'$. The generic fiber $\Exc^n(A,X) := \Exc^n(\calA,\calX)_{k'}$ is a closed subvariety of $X'$. The key properties we will require are given in the following two lemmas.

\begin{Lemma}[ {\cite[Proposition 3.2.1]{WissonThesis} }]\label{lem:Wisson}
	Suppose $L$ is a separable extension of $k'$. Then $$X(L) \cap p^nA(L) \subset \Exc^n(A,X)(L)\,.$$
\end{Lemma}

\begin{Corollary}
	Suppose $L$ is the completion of $k'$ at a discrete valuation. Then $$X(L) \cap p^nA(L) \subset \Exc^n(A,X)(L)\,.$$
\end{Corollary}

\begin{proof}
	Let $f:X_n \to X$ be the pullback of $[p^n]:A \to A$ and let $L^h$ be the henselization of $k'$ at $w$. Then $L^h$ is a separable extension of $k'$ by \cite[Lemma 3.1]{PoonenVoloch} and so $f(X_n(L^h)) = X(L^h) \cap p^nA(L^h) \subset \Exc^n(A,X)(L^h)$ by the lemma. Let $x = f(y) \in X(L) \cap p^nA(L) = f(X_n(L))$. By Greenberg's approximation theorem \cite{Greenberg}, there is a sequence of points in $X_n(L^h)$ converging to $y$ in $X_n(L)$. The image of this sequence is contained in $f(X_n(L^h)) \subset \Exc^n(A,X)(L^h)$ and converges to $x$. Since $\Exc^n(A,X)(L)$ is a closed subset of $A(L)$ containing $\Exc^n(A,X)(L^h)$ we must have $x \in \Exc^n(A,X)(L)$.
\end{proof}

We note that Lemma~\ref{lem:Wisson} also allows us to identify exceptional subschemes of $X$ which contain intersections with cosets of $p^nA(L)$. Given $Q \in \calA'(U') = A(k')$, let $X^{+Q}$ denote the translate of $X$ by $Q$. The lemma gives
\[
	X(L) \cap (Q + p^nA(L)) \subset \Exc^n(A,X^{-Q})^{+Q}(L)\,.
\]

\begin{Lemma}[R\"ossler]\label{lem:Rossler}
	Suppose the reduced subscheme of $X_\kbar$ is not a finite union of cosets in $A_\kbar$. Let $u \in U$ be a closed point and $k_u$ the completion of $k$ at $u$. Then there exists $m \ge 1$ such that for all $Q \in A(k_u)$ we have $\Exc^m(A,X^{+Q})$ is not Zariski dense in $X'^{+Q}$. 
\end{Lemma}

\begin{proof}
	By the Manin-Mumford conjecture \cite[Theorem 3.7]{PinkRossler} the hypothesis on $X_\kbar$ implies that the torsion points of $A_\kbar$ are not Zariski dense in $X_\kbar^{+Q}$ for any $Q$. By~\cite[Corollary 4.5]{RosslerML} this implies that there exists an $m$ such that for all $Q \in A(k_u)$ the sets
\begin{equation}\label{eq:Exc}
\Exc^m(\calA,\calX^{+Q})_{u'}(\Fbar) = \{P \in \calX^{+Q}(\Fbar) : P \text{ lifts to an element of } \calX^{+Q}(u'_m) \cap p^m\calA(u'_m)\}
\end{equation}
	 are not Zariski dense in $\calX^{+Q}_{u'}$. Here $u' \in U'$ is a closed point above $u$ and $u'_m$ is its $m$-th infinitesimal neighborhood. The equality between the two sets in~\eqref{eq:Exc} is noted in \cite[Proof of Proposition 3.1]{RosslerML}. Since $\calX^{+Q}$ is integral, this implies that $\Exc^m(A,X^{+Q}) = \Exc^m(\calA,\calX^{+Q})_{k'}$ is not Zariski dense in $X'^{+Q} = (\calX'^{+Q})_{k'}$.
\end{proof}

\subsection{The adelic Mordell-Lang conjecture over function fields}

\begin{Theorem}\label{thm:AML}
	Let $A$ be an abelian variety over a global function field $k$. Assume that $A_\kbar$ has no positive dimensional isotrivial quotient. Let $X \subset A$ be a closed subvariety. Then there exists a finite union of cosets $Y$ contained in $X$ such that
	$$X(\A_k) \cap \Sel(A) \subset Y(\A_k)\,.$$
\end{Theorem}

In the proof we will use the notion of an $n$-covering of $A$. This is an fppf torsor $\phi_T:T \to A$ under $A[n]$ that after base extension to a separable closure becomes isomorphic to the base extension of $[n] : A \to A$. The elements $\tau \in \Sel^n(A)$ are represented by locally soluble $n$-coverings of $A$. %Moreover, for every $n \ge 1$ one has that $\Sel(A) \subset \bigcup_{\tau \in \Sel^n(A)}\phi_U(U(\A_k))_\bullet$. See \cite[Lemma 5.4]{PoonenVoloch}.

\begin{proof}
	By Lemma~\ref{lem:Y'}, we are reduced to showing that there is a finite separable extension $K/k$ and a subvariety $Y \subset X_K$ such that $(Y_{\Kbar})_\textup{red}$ is a finite union of cosets and $X(\A_K) \cap \Sel(A_K) \subset Y(\A_K)$. Thus we may pass to a finite separable extension and assume that all irreducible components of $X$ are geometrically irreducible and that the reduced subscheme of at least one irreducible component of $X$ is not a coset.
	
	Let $u \in U$. Let $m = m(u)$ be the maximum of the integers given by R\"ossler's Lemma~\ref{lem:Rossler} as we range over the irreducible components of $X$ whose reduced subschemes are not cosets in $A$. Let $\pi_i : T_i \to A$, $i = 1,\dots,s$ be $p^m$-coverings of $A$ representing the finitely many elements of $\Sel^{p^m}(A)$. By \cite[Lemma 5.4]{PoonenVoloch} we have that $\Sel(A) \subset \bigcup_i \pi_i(T_i(\A_k))$. For each $i$, choose a point $t_i \in T_i(k_{u}^h)$ where $k_{u}^h$ denotes the Henselization of $k$ at $u$. Let $L/k$ be a finite separable extension over which all of the $t_i$ are defined. We may choose $L$ to have a prime $\mathfrak{u}$ of degree $1$ over $u$, so that $a_i = \pi_i(t_i) \in A(L_{\mathfrak{u}}) = A(k_{u})$. 
	
	Let $W \subset X$ be an irreducible component whose reduced subscheme is not a coset. By Lemma~\ref{lem:Rossler} we have that $\Exc^m(A,W^{-a_i})$ is not Zariski dense in $W^{-a_i}$, for all $i$. Then, for every $i$, the translate $\Exc^m(A,W^{-a_i})^{+a_i}$ is not Zariski dense in $W$. Note that each $\Exc^m(A,W^{-a_i})^{+a_i}$ is defined over some finite constant extension of the finite separable extension $L/k$. Let $Z$ be the union of the Galois conjugates of all of the $\Exc^m(A,W^{-a_i})^{+a_i}$ for $i = 1,\dots,s$. Then $Z$ is a subvariety of $W$. Moreover, since $W$ is geometrically irreducible, we have that $Z$ is not Zariski dense in $W$.
	
	By Lemma~\ref{lem:Wisson} and its corollary, for any prime $w$ of $L$ above a prime $v$ of $k$ and for all $i= 1,\dots,s$, we have
	\[
		W(k_v) \cap \pi_i(T_i(k_v)) \subset W(L_w) \cap \pi_i(T_i(L_w)) = W(L_w) \cap \left(a_i + p^m(A(L_w))\right) \subset Z(L_w)\,,
	\]
	and so
	\[
		W(k_v) \cap \pi_i(T_i(k_v)) \subset W(k_v) \cap Z(L_w) = Z(k_v)\,.
	\]
	
	Let $X_1 \subset X$ be the union of the irreducible components of $X$ whose reduced subschemes are cosets, together with the subvarieties $Z = Z(W)$ constructed for each irreducible component $W \subset X$ whose reduced subscheme is not a coset. For every prime $v$ of $k$ and every $i = 1,\dots,s$, we have
	\[
		X(k_v) \cap \pi_i(T_i(k_v)) \subset X_1(k_v)\,.
	\]
	Since $\Sel(A) \subset \bigcup_{i} \pi_i(T_i(\A_k))$, it follows that $X(\A_k) \cap \Sel(A) \subset X_1(\A_k)$. 
	
We now repeat the argument with $X_1$ in place of $X$ and proceed by induction to construct a sequence of varieties $X_1/k_1,X_2/k_2,\dots$, defined over finite separable extensions $k = k_1 \subset k_2 \subset \dots $, such that $X_{j+1} \subset (X_{j})_{k_{j+1}}$ and 
	$X(\A_{k_j}) \cap \Sel(A_{k_j}) \subset X_j(\A_{k_j})$ for all $j$. Note that $X_{j+1}$ is not Zariski dense in $X_j$ unless the reduced subscheme of $(X_j)_\kbar$ is a finite union of cosets. So by noetherianity the sequence must eventually stabilize at some $X_N$ such that the reduced subscheme of $(X_N)_{\overline{k}}$ is a finite union of cosets. As noted in the first paragraph, we may conclude by applying Lemma~\ref{lem:Y'} to $X_N/k_N$.
\end{proof}

\section{The Brauer Set}\label{sec:Br}		
	
	Let $X$ be a proper variety over a global field $k$. The Brauer group of $X$ is the \'etale cohomology group $\Br(X) := \HH^2(X,\G_m)$. There is a pairing
	\begin{equation}\label{eq:Br}
	 	\langle \,,\,\rangle \colon X(\A_k) \times \Br(X) \to \Q/\Z\,
	\end{equation}
	defined as
	\[
		\langle (x_v),\alpha\rangle = \sum_{v \in \Omega_k} \operatorname{inv}_v (x_v^*\alpha)\,
	\]
	where $x_v^*$ is the map $\Br(X) = \HH^2(X,\G_m) \to \HH^2(\Spec(k_v),\G_m) = \Br(k_v)$ induced by the $k_v$-point $x_v : \Spec(k_v) \to X$, and $\operatorname{inv}_v : \Br(k_v) \hookrightarrow \Q/\Z$ is the invariant map of local class field theory. 
	
	Define $X(\A_k)^{\Br}$ to be the subset of $X(\A_k)$ pairing trivially with all elements of $\Br(X)$. For $\alpha \in \Br(X)$, the map $\langle \,\,,\alpha \rangle : X(\A_k) \to \Q/\Z$ is locally constant. Hence there is an induced pairing on $X(\A_k)_\bullet \times \Br(X)$. We define $X(\A_k)^{\Br}_\bullet$ to be the subset of $X(\A_k)_\bullet$ pairing trivially with all elements of $\Br(X)$. Equivalently, $X(\A_k)_\bullet^{\Br}$ is the image of $X(\A_k)^{\Br}$ in $X(\A_k)_\bullet$.
\subsection{The Brauer set and the Selmer group}\label{sec:BrSel}

		The Hochschild-Serre spectral sequence in Galois cohomology (see \cite[(2.23)]{Skorobogatov}) gives
\begin{equation}\label{eq:HS}
	\Br(k) \to \ker(\Br(X) \to \Br(X_{k^\sep})) \stackrel{r}\to \HH^1(k,\Pic_X(k^\sep)) \to 0\,,
\end{equation}
where the last term is $\HH^3(k,\G_m) = 0$ since $k$ is a global field.  Let $\Br_0(X)$ denote the image of $\Br(k) \to \Br(X)$. By exactness we get a map $\HH^1(k,\Pic(X_{k^\sep})) \to \Br(X)/\Br_0(X)$. Composing this with the map induced by the inclusion $\Pic^0(X_{k^\sep}) \subset \Pic(X_{k^\sep})$ gives a map 
\begin{equation}\label{eqPic0}
	\HH^1(k,\Pic^0(X_{k^\sep})) \to \Br(X)/\Br_0(X). 
\end{equation}
We define $\Br_{1/2}(X) \subset \Br(X)$ to be the subgroup consisting of all elements whose image in $\Br(X)/\Br_0(X)$ lies in the image of~\eqref{eqPic0}. The subgroup $\Br_0(X)$ pairs trivially with all elements of $X(\A_k)$, so the restriction of~\eqref{eq:Br} to $\Br_{1/2}(X)$ yields a map
	\[
	X(\A_k)_\bullet \stackrel{\textup{BM}}\to (\Br_{1/2}(X)/\Br_0(X))^D
	\]
	where the ${}^D$ means $\Hom(-,\Q/\Z)$. Let $X(\A_k)^{\Br_{1/2}}$ denote the subset of $X(\A_k)$ pairing trivially under~\eqref{eq:Br} with all elements in $\Br_{1/2}(X)$, and define $X(\A_k)_\bullet^{\Br_{1/2}}$ similarly. Equivalently, $X(\A_k)^{\Br_{1/2}}_\bullet$ is the subset of $X(\A_k)_\bullet$ mapping to $0$ in $(\Br_{1/2}(X)/\Br_0(X))^D$ under the map $\textup{BM}$.
\pagebreak
\begin{Lemma}\label{lem:1/2}
	Let $A$ be an abelian variety over a global field $k$. Then $A(\A_k)^{\Br_{1/2}}_\bullet = \Sel(A)$.
\end{Lemma}	
	
\begin{proof}
As in the proof of \cite[Theorem E]{PoonenVoloch}, the Cassels dual exact sequence and the BM map sit in a commutative diagram with exact row:
		\[ 
			\xymatrix{
				0 \ar[r] &\Sel(A) \ar[r] & A(\A_k)_{\bullet} \ar[dr]_{\textup{BM}} \ar[r]^-{\textup{Tate}} & \HH^1(k,\Pic^0(A_{k^\sep}))^D \\
				&&& \left(\frac{\Br_{1/2}(A)}{\Br_0(A)}\right)^D \ar[u]
			}
		\]
		where `$\textup{Tate}$' is induced by the sum of the local Tate pairings $A(k_v) \times \HH^1(k_v,\Pic^0(A_{k^\sep})) \to \Q/\Z$. The vertical map, which comes from~\eqref{eqPic0}, is injective by the definition of $\Br_{1/2}(A)$. So we have $\Sel(A) = \ker(\textup{Tate}) = \ker(\textup{BM}) = A(\A_k)^{\Br_{1/2}}_\bullet$.
\end{proof}

We now prove the analogue of Lemma~\ref{lem:T} for $\Br_{1/2}$.
	
\begin{Lemma}\label{lem:BrSel}
	Let $C \subset A$ be a coset and assume that $k$ has no real primes. Then $C(\A_k)^{\Br_{1/2}} = C(\A_k) \cap A(\A_k)^{\Br_{1/2}}\,.$
\end{Lemma}

\begin{proof}
	The inclusion $\iota : C \to A$ induces a commutative diagram
		\[
			\xymatrix{ \HH^1(k,\Pic^0(C_{k^\sep})) \ar[r]  & \Br(C)/\Br_0(C) \\
			\HH^1(k,\Pic^0(A_{k^\sep})) \ar[r] \ar[u]^{\iota^*} & \Br(A)/\Br_0(A) \ar[u]^{\iota^*}
			}
		\]
		from which it follows that $\iota$ gives a map $\iota^*:\Br_{1/2}(A) \to \Br_{1/2}(C)$. By functoriality of the pairing~\eqref{eq:Br} we have that $C(\A_k)^{\Br_{1/2}} \subset C(\A_k)^{\iota^*\Br_{1/2}(A)} = C(\A_k) \cap A(\A_k)^{\Br_{1/2}}$. Thus $C(\A_k)^{\Br_{1/2}} \subset A(\A_k)^{\Br_{1/2}}$. For the reverse containment we must show that the image of $C(\A_k)^{\Br_{1/2}}$ in $A(\A_k)$ contains $C(\A_k) \cap A(\A_k)^{\Br_{1/2}}$.
		
		Suppose $C = a + A'$ with $A' \subset A$ an abelian subvariety. Then $\Pic^0_A = A^\vee$ and $\Pic^0_C = A'^\vee$ are the dual abelian varieties and there is an exact sequence $0 \to A''^\vee \to A^\vee \to A'^\vee \to 0$, where $A'' = A/A'$. Note that $\HH^1(k,A'^\vee) = \HH^1(k,\Pic^0(C_{k^\sep}))$. Galois cohomology gives a diagram with exact rows,
	\begin{equation}\label{eq:real}
		\xymatrix{
			\HH^1(k,A^\vee) \ar[d]\ar[r]& \HH^1(k,A'^\vee) \ar[r]\ar[d] & \HH^2(k,A''^\vee) \ar[r]\ar[d] & \HH^2(k,A^\vee) \ar[d] \\
			\displaystyle\bigoplus_{v \in \Omega_k} \HH^1(k_v,A^\vee) \ar[r] & \displaystyle\bigoplus_{v \in \Omega_k} \HH^1(k_v,A'^\vee) \ar[r]& \displaystyle\bigoplus_{v \in \Omega_k} \HH^2(k_v,A''^\vee) \ar[r] & \displaystyle\bigoplus_{v \in \Omega_k} \HH^2(k_v,A^\vee)
			}
	\end{equation}
	For complex and nonarchimedean primes $v$, the groups $\HH^2(k_v,A''^\vee)$ are $0$ by \cite[I.3.4, I.3.6, and III.7.8]{MilneADT}. Moreover, $\HH^2(k,A''^\vee) \simeq \bigoplus_{v \textup{ real}}\HH^2(k_v,A''^\vee)$ by \cite[Theorem I.6.26]{MilneADT} in the number field case, and in the function field case by \cite[Lemma 3.3]{GA-THassePrinciple}. Since we assume $k$ has no real primes, this implies that the map
	\[
		\iota^* : \Br_{1/2}(A)/\Br_0(A) \simeq \HH^1(k,A^\vee) \to \HH^1(k,A'^\vee) \simeq \Br_{1/2}(C)/\Br_0(C)
	\]
	is surjective and, hence, that $C(\A_k)^{\Br_{1/2}} = C(\A_k)^{\iota^*\Br_{1/2}(A)}$ completing the proof. 
\end{proof}

\begin{Remark}\phantom{.}
	\begin{enumerate}
		\item The proof of the previous lemma shows that $\Br_{1/2}(A)/\Br_0(A) \to \Br_{1/2}(C)/\Br_0(C)$ is surjective. We do not know if $\Br(A)/\Br_0(A) \to \Br(C)/\Br_0(C)$ is always surjective. Nevertheless, it will follow from Theorem~\ref{thm:notranscendental} below that $C(\A_k)^{\Br} = C(\A_k) \cap A(\A_k)^{\Br}$.
		\item When $k$ has real primes, the map $\Br_{1/2}(A)/\Br_0(A) \to \Br_{1/2}(C)/\Br_0(C)$ need not be surjective, but one can still show that $C(\A_k)^{\Br_{1/2}}$ and $C(\A_k)\cap A(\A_k)^{\Br_{1/2}}$ have the same image in $A(\A_k)_\bullet$.
	\end{enumerate}
\end{Remark}

\subsection{The Brauer set for torsors under abelian varieties}

We now prove that, for a torsor under an abelian variety $A$ over a global field, Brauer-Manin is the only obstruction to weak approximation if and only if $\Sha(A)$ contains no nontrivial divisible elements. That the Brauer group controls the Hasse principle when $\Sha(A)_\textup{div}$ is trivial was already known to Manin when he first introduced the obstruction~\cite{Manin}, at least over number fields. The analogous statement for weak approximation was given in~\cite{Wang} and also in~\cite[Theorem 1]{Harari}. The converse of these results over number fields was proved in \cite{CreutzBCyr}. To extend this to all global fields requires that we develop some aspects of the descent theory for abelian torsors as described in \cite[Chapter 6]{Skorobogatov} in the function field setting (See Lemma~\ref{lem:descenttheory}).

\begin{Theorem}\label{thm:notranscendental}	
	Let $T$ be a torsor under an abelian variety $A$ over a global field $k$. Then:
	\begin{enumerate}
		\item\label{it:t1} $T(\A_k)^{\Br} = T(\A_k)^{\Br_{1/2}}$.
		\item\label{it:t2} $T(\A_k)^{\Br} \ne \emptyset$ if and only if $T$ represents an element of $\Sha(A)_\textup{div}$.
		\item\label{it:t3} Assume $T(\A_k)^{\Br} \ne \emptyset$. Then $T(k)$ is dense in $T(\A_k)^{\Br}_\bullet$ if and only if $\Sha(A)_\textup{div} = 0$.
	\end{enumerate}
\end{Theorem}

\begin{Corollary}\label{cor:notranscendental}
	If $A$ is an abelian variety over a global field $k$ then $A(\A_k)^{\Br}_\bullet = \Sel(A)$. Moreover, the image of $A(k)$ in $A(\A_k)^{\Br}_\bullet$ is dense if and only if $\Sha(A)_\textup{div}= 0$.
\end{Corollary}

For the proof we will require the following two lemmas involving $m$-coverings of $T$. Such an object is an fppf torsor $\phi : U \to T$ under $A[m]$ which after base extension to $k^\sep$ becomes isomorphic to multiplication by $m$ on $A_{k^\sep}$.

\begin{Lemma}\label{lem:descenttheory}
	Let $x \in T(\A_k)^{\Br_{1/2}}$. Then for all $m$ there exists an $m$-covering $\phi:U \to T$ with $x \in \phi(U(\A_k))$.
\end{Lemma}

\begin{proof}
	In the number field case, this follows from the descent theory for abelian torsors \cite[6.1.2]{Skorobogatov}. We expect a similar theory holds over global function fields, but do not know of a reference. The most delicate part is the proof of the existence of an $m$-covering. One can then use arithmetic duality theorems (e.g., Poitou-Tate sequence) to find a locally soluble $m$-covering. In the number field case the proof of existence involves computations with cocycles in Galois cohomology, so does not extend to the case of global function fields where one must work with fppf cohomology. Instead we give the following argument particular to our situation (which also works over number fields, though we will assume $k$ is a global function field to avoid having to deal with archimedean primes).
	
	Let $K/k$ be a finite separable extension such that $T(K) \ne \emptyset$. Then $B = \Res_{K/k}(A_K)$ is an abelian variety, and $A$ is an abelian subvariety of $B$ via the canonical map $A \to \Res_{K/k}(A_K)$. Since $T$ is a torsor under $A$, any choice of base point in $T(K)$ induces an isomorphism $T_K \simeq A_K$. We can identify $T$ with its image under the map $T \to \Res_{K/k}(T_K) \simeq \Res_{K/k}(A_K) = B$ induced by a choice of base point. Then $T$ is a coset of $A$ inside $B$. Let $Q = B/A$.
	
	By Lemma~\ref{lem:BrSel} and Lemma~\ref{lem:1/2} we have that $x\in T(\A_k)^{\Br_{1/2}} \subset B(\A_k)^{\Br_{1/2}} = \Sel(B) = \varprojlim_m \Sel^m(B)$. This element $x \in \Sel(B)$ is represented by a compatible family of $m$-coverings $\rho_m : B_m \to B$ in $\Sel^m(B)$ (here compatible means that for any $m,m' \ge 1$ we have that $\rho_{mm'}$ factors through $\rho_m$). The fact that this represents $x$ means that $\varprojlim_{m}B_m(\A_k)$ contains a compatible system of points $x_m \in B_m(\A_k)$ such that $\rho_m(x_m) = x$. Consider the pullback $\phi_m : U_m \to T$ of $\rho_m$ along the closed immersion $T \to B$. Then $U_m \to B_m$ is a closed immersion. Since $x \in T(\A_k)$, we have $x_m \in U_m(\A_k) \subset B_m(\A_k)$. Note that $U_m \to T$ is a $B[m]$-torsor, whereas an $m$-covering of $T$ is an $A[m]$-torsor.
		
	Since $x \in T(\A_k)$, the image of $x$ under $\Sel(B) \to \Sel(Q)$ is a $k$-rational point $q \in Q(k) \subset \Sel(Q)$ represented by the compatible family of $m$-coverings $B_m/A \to B/A = Q$ (these are $Q[m]$-torsors). The point $q \in Q(k) \subset \Sel(Q)$ is also represented by the compatible family of $m$-coverings $\delta_{m,q} : Q \ni u \mapsto mu + q \in Q$. We have $\delta_{m,q}^{-1}(q) = Q[m]$. For any prime $v$ of $k$, we have $\varprojlim_m Q(k_v)[m] = 0$, so the only compatible system of lifts $q_m \in Q(\A_k)$ with $\delta_{m,q}(q_m) = q$ is with $q_m = 0_Q \in Q(k) \subset Q(\A_k)$, for all $m$. By the injectivity of $\Sel(Q) \to Q(\A_k)$ proved in Proposition~\ref{lem:infords}, each $B_m/A \to Q$ is isomorphic to $\delta_{m,q}$. The images $y_m \in (B_m/A)(\A_k)$ of $x_m$ are compatible lifts of $q$. It follows that $y_m$ are $k$-rational points of $B_m/A$. 
	
	Let $\Gamma_m$ denote the fiber of $B_m/A \to Q$, so that $y_m \in \Gamma_m(k)$. We note that $\Gamma_m \simeq Q[m]$ as group schemes if we choose $y_m$ as the distinguished point of $\Gamma_m$. Since $T$ maps to $q$ under $B \to Q$, we have that the map $U_m \to B_m \to B_m/A$ factors through $\Gamma_m \to B_m/A$. By the universal property of pullbacks, we get a map $U_m \to T \times_q \Gamma_m$. This gives rise to a commutative diagram of pullback squares defining $V_m$.
	\[
		\xymatrix{	
		V_m  \ar[d]\ar[r] &T \times_q y_m \ar[d]\ar[r] & y_m \ar[d] \\
		U_m \ar[r]\ar[dr]_{\phi_m} & T\times_q \Gamma_m \ar[d]\ar[r] & \Gamma_m \ar[r]\ar[d] & B_m/A \ar[d] \\
		& T \ar[r] & q \ar[r] & Q
		}
	\]
	Since $y_m \to \Gamma_m$ is a closed immersion we have that $V_m \to U_m$ is a closed immersion. Moreover, since $x_m$ maps to $y_m$, we have that $x_m \in V_m(\A_k)$ and so $x \in \phi_m(V_m(\A_k))$.
	
	We claim that the restriction of $\phi_m$ to $V_m$ is an $m$-covering of $T$. For this we may work geometrically, replacing $T$ with $A$ and $\rho_m$ with multiplication by $m$ on $B$. Then $B_m = B$, $U_m = \pi^{-1}(Q[m])$ and $V_m = \pi^{-1}(0_Q) = A$, where $\pi : B \to Q$ is the quotient map. The restriction of $\phi_m$ to $V_m = A$ is multiplication by $m$ on $A$. So it is an $m$-covering.
\end{proof}

\begin{Remark}
	For a coset $C$ in an abelian variety $A$ one could define $\Sel(C) \subset C(\A_k)$ to be the set of adelic points which lift to an $m$-covering, for all $m$. The proof of the previous lemma then shows $\Sel(C) = C(\A_k) \cap \Sel(A)$ analogously to Lemma~\ref{lem:T} and Lemma~\ref{lem:BrSel}.
\end{Remark}

\begin{Lemma}\label{lem:Bcyrlem}
	Let $\alpha \in \Br(T)/\Br_0(T)$ be an element of order $n$. There exists a power $m$ of $n$ such that for any $m$-covering $\phi:U \to T$ we have that $\phi^*(\alpha) = 0$ in $\Br(U)/\Br_0(U)$.
\end{Lemma}

\begin{proof}
	In the number field case this is~\cite[Lemma 13]{CreutzBCyr}, which was inspired by~\cite[Lemma 4.6]{CViray}. Essentially the same proof now works in the function field case using recent results of Skorobogatov and Yang concerning $p$-primary torsion in the Brauer group of abelian varieties. Specifically, by \cite[Lemma 3.1]{Skorob} (which is a consequence of \cite[Theorem 1.7]{Yang}), for any integer $\ell$, multiplication by $n$ on $A_{\kbar}$ induces multiplication by $n^2$ on $\HH^1_{\textup{fppf}}(A_{\kbar},\mu_\ell)$. From the exact sequence
	\[
		0 \to \NS(A_{\kbar})/\ell \to \HH^1_{\textup{fppf}}(A_{\kbar},\mu_\ell) \to \Br(A_{\kbar})[\ell] \to 0
	\]
	we deduce that multiplication by $n$ induces multiplication by $n^2$ on the finitely generated abelian group $\NS(A_{\kbar})$ and the torsion abelian group $\Br(A_{\kbar})$. The base extension gives an isomorphism $\NS(A_{k^\sep}) \to \NS(A_{\kbar})$ (as noted in the proof of~\cite[Corollary 1.4]{Skorob}) and an injection $\Br(A_{k^\sep}) \to \Br(A_{\kbar})$ by \cite[Corollary 3.4]{DAddezio}. So multiplication by $n$ on $A$ induces multiplication by $n^2$ on $\NS(A_{k^\sep})$ and on $\Br(A_{k^\sep})$. This is the function field analogue of~\cite[Lemma 12]{CreutzBCyr} (which had already appeared in~\cite{Berko}).
	
	Let $\pi : U \to T$ be an $n^4$-covering. To prove the lemma it will suffice to show that $\pi^*\alpha = 0$. The $n^4$-covering $\pi$ factors as $U =: U_4 \stackrel{\pi_4} \to U_{3} \stackrel{\pi_{3}}\to U_{2} \stackrel{\pi_{2}} \to U_{1} \stackrel{\pi_1}\to U_0 := T$ where each $\pi_i$ is an $n$-covering of $U_{i-1}$. For all $i$, we have $(U_i)_{k^\sep} \simeq A_{k^\sep}$ and after this base extension the torsor $\pi_i$ is isomorphic to multiplication by $n$ on $A_{k^\sep}$. For any $i$,~\eqref{eq:HS} gives an exact sequence
	\[
		\HH^1(k,\Pic((U_i)_{k^\sep})) \hookrightarrow \Br(U_i)/\Br_0(U_i) \to \Br((U_i)_{k^\sep})\,.
	\]
Our element $\alpha$ has order $n$, so its image in $\Br((U_0)_{k^\sep}) \simeq \Br(A_{k^\sep})$ will be annihilated by $\pi_1^*$ since this acts as multiplication by $n^2$ there. Hence $\pi_1^*\alpha$ lies in $\HH^1(k,\Pic((U_1)_{k^\sep}))[n]$.
	
	There is an exact sequence of $\Gal(k^\sep/k)$-modules
	\[
		0 \to \Pic^0(A_{k^\sep}) \to \Pic(A_{k^\sep}) \to \NS(A_{k^\sep}) \to 0\,,
	\]
	and similarly for the $(U_i)_{k^\sep}$, from which we obtain exact sequences
	\begin{equation}\label{eq:gr}
		\HH^0(k,\NS((U_i)_{k^\sep})) \to \HH^1(k,\Pic^0((U_i)_{k^\sep})) \to \HH^1(k,\Pic((U_i)_{k^\sep})) \to \HH^1(k,\NS((U_i)_{k^\sep})).
	\end{equation}
	By~\cite[Lemma 10]{CreutzBCyr} we have isomorphisms of $\Gal(k^\sep/k)$-modules $\NS((U_i)_{k^\sep}) \simeq \NS(A_{k^\sep})$ and $\Pic^0((U_i)_{k^\sep}) \simeq \Pic^0(A_{k^\sep})$ (The proof of this statement given there works over any field). Hence, the first, second and fourth terms in~\eqref{eq:gr} can be identified with the corresponding group for $A_{k^\sep}$. Under this identification the $\pi_i$ induce multiplication by $n$ and $n^2$ on $\HH^1(k,\Pic^0(A_{k^\sep}))$ and $\HH^j(k,\NS(A_{k^\sep}))$, respectively. Thus, for $i >0$, we obtain commutative diagrams with exact rows:
	\[
		\xymatrix{
		\HH^0(k,\NS(A_{k^\sep})) \ar[r]^{\delta_{i-1}}\ar[d]^{n^2} & \HH^1(k,\Pic^0(A_{k^\sep})) \ar[r]^{\epsilon_{i-1}}\ar[d]^n& \HH^1(k,\Pic((U_{i-1})_{k^\sep})) \ar[r]^{\gamma_{i-1}}\ar[d]^{\pi_i^*} & \HH^1(k,\NS(A_{k^\sep})) \ar[d]^{n^2}\\
		\HH^0(k,\NS(A_{k^\sep})) \ar[r]^{\delta_{i}} & \HH^1(k,\Pic^0(A_{k^\sep})) \ar[r]^{\epsilon_{i}}& \HH^1(k,\Pic((U_{i})_{k^\sep})) \ar[r]^{\gamma_i} & \HH^1(k,\NS(A_{k^\sep}))\,. \\
		}
	\]
	For $i = 1,\dots, 4$, let $\alpha_i = \pi_i^*\cdots \pi_1^*\alpha$. Note that $n\alpha_i = 0$ for all $i$. Using this and commutativity of the diagram above we get $\gamma_2(\alpha_2) = \gamma_2(\pi_2^*\alpha_1) = n^2 \gamma_1(\alpha_1) = \gamma_1(n^2\alpha_1) = 0$. So $\alpha_2 = \epsilon_2(\beta)$ for some $\beta \in \HH^1(k,\Pic^0(A_{k^\sep}))$ by exactness of the rows. Similarly, we have that $\epsilon_2(n\beta) = n\epsilon_2(\beta) = n\alpha_2 = 0$, so $n\beta = \delta_2(\gamma)$ for some $\gamma \in \HH^0(k,\NS(A_{k^\sep}))$. Then
	\[
		\pi^*\alpha = \pi_4^*\pi_3^*\epsilon_2(\beta) \stackrel{\star}= \epsilon_4(n^2\beta) = \epsilon_4(n\delta_2(\gamma)) \stackrel{\star} = \epsilon_4(\delta_3(n^2\gamma)) = n\epsilon_4(n\delta_3(\gamma)) \stackrel{\star}= n\epsilon_4(\delta_4(n^2\gamma)) = 0\,,
	\]
	where the equalities labelled with a $\star$ come from commutativity of the diagram and the final equality comes from exactness of the rows.
\end{proof}

\begin{proof}[Proof of Theorem~\ref{thm:notranscendental}]
	First note that all of the statements hold trivially if $T(\A_k)=\emptyset$, so we will assume $T(\A_k)\ne\emptyset$ which means that $T$ represents an element of $\Sha(A)$. 

	Since $\Br_{1/2}(T) \subset \Br(T)$ we have $T(\A_k)^{\Br} \subset T(\A_k)^{\Br_{1/2}}$. Suppose $x \in T(\A_k)^{\Br_{1/2}}$ and let $\alpha \in \Br(T)$. To prove~\eqref{it:t1} we will show that $x$ and $\alpha$ pair trivially in~\eqref{eq:Br}. Let $n$ be the order of $\alpha$ in the torsion abelian group $\Br(T)$. By Lemma~\ref{lem:Bcyrlem} there is a power $m$ of $n$ such that for any $m$-covering $\phi:U \to T$ we have that $\phi^*(\alpha) \in \Br_0(U) \subset \Br(U)$. By Lemma~\ref{lem:descenttheory} there is an $m$-covering $\phi:U \to T$ with $x = \phi(y) \in \phi(U(\A_k))$. For the pairings~\eqref{eq:Br} on $U$ and $T$ we then have 
	\begin{equation}\label{eq:func}
		\langle x , \alpha\rangle = \langle \phi(y),\alpha \rangle = \langle y,\phi^*\alpha \rangle = 0\,
	\end{equation}
	since $\phi^*\alpha \in \Br_0(U)$. So $\alpha$ is orthogonal to $x$.
		
	We now prove~\eqref{it:t2}. Suppose $d$ is the order of $T$ in $\Sha(A)$. For any $e \ge 1$ there is a commutative diagram
	\[
		\xymatrix{
			\Sel^{de}(A)\ar@{>>}[r]  \ar[d]^{e_*} & \Sha(A)[de] \ar[d]^{e}\\
			\Sel^d(A) \ar@{>>}[r] & \Sha(A)[d]
		}
	\]
	Let $\pi : T \to A$ be a $d$-covering representing a lift of $T$ to $\Sel^d(A)$. Any lift of $\pi$ to $\Sel^{de}(A)$ is represented by a $de$-covering of $A$ which factors through $\pi$, as an $e$-covering of $T$. Also, any $e$-covering of $T$ gives a $de$-covering of $A$ by composing with $\pi$. From this and the diagram above it follows that $T \in \Sha(A)_\textup{div}$ if and only if for every $e\ge 1$, there is an $e$-covering $\phi:U \to T$ with $U(\A_k) \ne \emptyset$. If $T(\A_k)^{\Br_{1/2}} \ne \emptyset$, then such $e$-coverings exist by Lemma~\ref{lem:descenttheory} and so $T$ must represent an element of $\Sha(A)_\textup{div}$. Conversely, suppose $T(\A_k)^{\Br_{1/2}} = \emptyset$. By compactness there are finitely many $\alpha_1,\dots,\alpha_r \in \Br_{1/2}(T)$ such that $T(\A_k)^{\{\alpha_1,\dots,\alpha_r\}} = \emptyset$. Let $m_i$ be the integers given by Lemma~\ref{lem:Bcyrlem} for each $\alpha_i$, and let $M$ be the product of the $m_i$. So, for any $M$-covering $\phi : U \to T$ we have $\phi^*\alpha_i \in \Br_0(U)$ for all $i$. If $T$ is divisible by $M$ in $\Sha(A)$, then there is an adelic point $x \in T(\A_k)$ which lifts to an $M$-covering of $T$. Arguing as with~\eqref{eq:func} we get $\langle x , \alpha_i \rangle = 0$, for all $i$. But then $ x \in T(\A_k)^{\{ \alpha_1,\dots,\alpha_r\}}$, contradicting our assumption.
	
	We now prove~\eqref{it:t3}. Suppose $T(\A_k)^{\Br}$ is nonempty. By~\eqref{it:t2} this is equivalent to assuming $T$ lies in $\Sha(A)_\textup{div}$. If $T(k) = \emptyset$, then $\Sha(A)_\textup{div} \ne 0$ (since $T$ is a nontrivial element there) and the statement holds. If $T(k) \ne \emptyset$, then $T \simeq A$ and so it suffices to prove~\eqref{it:t3} for $T = A$. By~\eqref{it:t1} and Lemma~\ref{lem:1/2} we have $A(\A_k)^{\Br}_\bullet = \Sel(A)$. From the exact sequence
	\[
		0 \to \varprojlim_n A(k)/nA(k)  \to \Sel(A) \to \varprojlim_n \Sha(A)[n] \to 0
	\]
	and the equality $\varprojlim_n A(k)/nA(k) = \overline{A(k)}$ given in \cite[Theorem E]{PoonenVoloch} we see that $\overline{A(k)} = A(\A_k)^{\Br}_\bullet$ if and only if $\varprojlim_n \Sha(A)[n] = 0$, which is if and only if $\Sha(A)_\textup{div} = 0$.
\end{proof}

\section{Proofs of the main theorems}\label{sec:proofs}

The following gives Theorems~\ref{thm:1} and~\ref{thm:2} of the introduction. Theorem~\ref{thm:2a} is an easy consequence of these two. Recall that $X(\A_k)_\circ = \prod_{v \in \Omega^\circ_k}X(k_v)$ is the product over the set $\Omega_k^\circ$ of nonarchimedean primes and $X(\A_k)_\circ^{\Br}$ denotes the image of $X(\A_k)^{\Br}$ in $X(\A_k)_\circ$.

\begin{Theorem}\label{thm:conditional}
	Let $A$ be an abelian variety over a global field $k$ and let $X$ be a closed subvariety of $A$. If $k$ is a global function field assume that $A_{\kbar}$ has no nonzero isotrivial quotient. If $k$ is a number field assume that $k$ is totally imaginary and Conjecture~\ref{conj:AML} holds for $X \subset A$. Then
	\begin{enumerate}
		\item The images of $\overline{X(k)}$ and $X(\A_k) \cap \overline{A(k)}$ in $A(\A_k)_\circ$ are equal, and
		\item\label{it:2} There is a finite collection of cosets $C_i = a_i + A_i \subset X$, $i = 1,\dots,r$, such that 
			\[
			X(\A_k)^{\Br}_\circ = \bigcup_{i = 1}^r (C_i(\A_k)_\circ^{\Br})\,.
			\]
			Moreover, $X(k)$ is dense in $X(\A_k)^{\Br}_\circ$ if and only if $\Sha(A_i)_\textup{div} = 0$ for all $i = 1,\dots,r$.
	\end{enumerate}
\end{Theorem}

\begin{Remark}
	In the theorem we must omit the complex primes, if we want to allow $X$ that are not geometrically connected. Here is an example with $X(k)$ not dense in $X(\A_k)^{\Br}_\bullet$. Suppose $k = \Q(i)$ and $X = P_1 \cup P_2 = \Spec(k \times k)$ is a pair of $k$-rational points. Let $(x_v) \in X(\A_k)$ be the adelic point with $x_v = P_1$ for all nonarchimedean primes and $x_v = P_2$ at the complex prime. Then $(x_v)$ and $P_1$ only differ at the complex prime. Since $P_1 \in X(k) \subset X(\A_k)^{\Br}$ and $\Br(\C) = 0$, we have that $(x_v) \in X(\A_k)^{\Br}$. But the image of $(x_v)$ in $X(\A_k)_\bullet$ is not contained in the image of $X(k)$.
\end{Remark}

\begin{proof}
	As in Section~\ref{sec:cosets}, we omit the subscript ${}_\circ$ from the notation, using $X(\A_k)$ to also denote its image in $A(\A_k)_\circ$ and similarly for the other subsets of $A(\A_k)$ considered.
	
	Theorem~\ref{thm:AML} (or the assumed Conjecture~\ref{conj:AML} in the number field case) gives a finite union of cosets $Y = \bigcup C_j$, say $C_j = a_j + A_j$, contained in $X$ such that $X(\A_k) \cap \Sel(A) \subset Y(\A_k)$. The first statement is then proved exactly as in the proof of Theorem~\ref{thm:1} given in Section~\ref{sec:AML}.

	By Lemma~\ref{lem:1/2}, $A(\A_k)^{\Br_{1/2}} = \Sel(A)$. So by Lemma~\ref{lem:Yi} and Lemma~\ref{lem:BrSel} we have
	\[X(\A_k) \cap A(\A_k)^{\Br_{1/2}} \subset Y(\A_k) \cap A(\A_k)^{\Br_{1/2}} \,\stackrel{\ref{lem:Yi}}\subset\, \bigcup_j \left(C_j(\A_k) \cap A(\A_k)^{\Br_{1/2}}\right) \,\stackrel{\ref{lem:BrSel}}{=}\, \bigcup_j C_j(\A_k)^{\Br_{1/2}}\,.\]
	By Theorem~\ref{thm:notranscendental}\eqref{it:t1} we have $C_j(\A_k)^{\Br_{1/2}} = C_j(\A_k)^{\Br}$, so we conclude that
	\[
		X(\A_k)^{\Br} \subset X(\A_k) \cap A(\A_k)^{\Br_{1/2}} \subset  \bigcup_j C_j(\A_k)^{\Br}
	\]
	This implies $X(\A_k)^{\Br} = \bigcup_j C_j(\A_k)^{\Br}$, since the reverse containment holds by functoriality of the Brauer pairing~\eqref{eq:Br}. Note that this equality continues to hold if we take the union only over those $j$ such that $C_j(\A_k)^{\Br} \ne \emptyset$. Excluding those cosets with empty Brauer set (and relabelling if needed) we have a finite collection of cosets $C_i = a_i + A_i$, $i=1,\dots,r$ with $C_i(\A_k)^{\Br} \ne \emptyset$ and $X(\A_k)^{\Br} = \bigcup_i C_i(\A_k)^{\Br}$.
	
	It remains to show that $\overline{X(k)} = X(\A_k)^{\Br}$ if and only if $\Sha(A_i)_\textup{div}=0$ for all $i = 1,\dots, r$.
	
	If $\Sha(A_i)_\textup{div}=0$ for all $i$, then by (the easy direction of) Theorem~\ref{thm:notranscendental}\eqref{it:t3} we have
	$$X(\A_k)^{\Br} = \bigcup_i C_i(\A_k)^{\Br} = \bigcup_i \overline{C_i(k)} \subset \overline{X(k)},$$
	which implies that $X(\A_k)^{\Br} = \overline{X(k)}$.
	
	For the converse suppose $\overline{X(k)} = X(\A_k)^{\Br} = \bigcup C_i(\A_k)^{\Br}$. We can assume that this set is nonempty, otherwise the collection of cosets is empty and there is nothing to prove. Let $C = a + A'$ be any of the cosets and let $x \in  C(\A_k)^{\Br} \subset X(\A_k)^{\Br} = \overline{X(k)} \ne \emptyset$. Since $X(k) = \bigcup C_i(k)$, we have $\overline{X(k)} = \bigcup \overline{C_i(k)}$. So there is a coset $C_i = a_i + A_i$ with $x \in \overline{C_i(k)}$. In particular, $C_i(k)\ne \emptyset$. Translating by a $k$-rational point we can assume $C_i = A_i$ is an abelian subvariety. Then $Y := C \cap A_i$ is geometrically a finite union of cosets contained in the abelian variety $A_i$ and $x \in Y(\A_k) \cap \overline{A_i(k)}$. By Lemma~\ref{lem:Y'}, there is a finite union of cosets $Y' \subset Y$ such that $Y(\A_k) \cap \overline{A_i(k)} = Y'(\A_k) \cap \overline{A_i(k)}$. Note that Conjecture~\ref{conj:AML} holds for $Y' \subset A$, since $Y'$ is already a finite union of cosets. By the first statement of the theorem (which we have already proved) we get that $Y'(\A_k) \cap \overline{A_i(k)} = \overline{Y'(k)} \subset \overline{C(k)}$. We conclude that $x \in \overline{C(k)}$. Since $x$ was arbitrary this shows that $C(\A_k)^{\Br} = \overline{C(k)}$. By Theorem~\ref{thm:notranscendental}  this implies that $\Sha(A')_\textup{div} = 0$.
	\end{proof}

\end{document}